\documentclass{amsart}
\usepackage[utf8]{inputenc}
\usepackage{amsmath,amssymb,amsthm}
\usepackage{todonotes}
\usepackage{mathrsfs}
\usepackage[hidelinks]{hyperref}
\usepackage{xcolor}
\usepackage{graphicx}
\usepackage[percent]{overpic}

\hypersetup{
    linktocpage,
    colorlinks,
    linkcolor={magenta!100!black},
    citecolor={cyan!100!black},
    urlcolor={yellow!100!black}
}

\numberwithin{equation}{section}

\newcommand{\F}{\mathcal{F}}
\newcommand{\N}{\mathbb{N}}
\newcommand{\R}{\mathbb{R}}
\newcommand{\ZZ}{\mathbb{Z}}
\newcommand{\RR}{\R}
\newcommand{\C}{\mathbb{C}}
\newcommand{\HH}{\mathbb{H}}
\newcommand{\sG}{\mathsf{G}}
\newcommand{\Es}{E^s}
\newcommand{\Ec}{E^c}
\newcommand{\Eu}{E^u}
\newcommand{\Ecs}{E^{cs}}
\newcommand{\Ecu}{E^{cu}}
\newcommand{\eps}{\varepsilon}
\newcommand{\RP}{{\mathbb{R}\mathbb{P}}}
\renewcommand{\a}{\mathfrak{a}}
\newcommand{\ap}{\mathfrak{a}^+}
\newcommand{\liesl}{\mathfrak{sl}_2(\mathbb{R})}

\DeclareMathOperator{\SL}{\mathsf{SL}}
\DeclareMathOperator{\PSL}{\mathsf{PSL}}
\DeclareMathOperator{\Gr}{\mathcal{G}}
\DeclareMathOperator{\GL}{\mathsf{GL}}
\DeclareMathOperator{\SO}{\mathsf{SO}}
\DeclareMathOperator{\dist}{dist}

\newtheorem{problem}{Problem}
\newtheorem{theorem}{Theorem}[section]
\newtheorem{teo}[theorem]{Theorem}
\newtheorem{lemma}[theorem]{Lemma}
\newtheorem{lem}[theorem]{Lemma}
\newtheorem{proposition}[theorem]{Proposition}
\newtheorem{cor}[theorem]{Corollary}
\newtheorem{prop}[theorem]{Proposition}

\theoremstyle{definition}
\newtheorem{rem}[theorem]{Remark}

\title[Quasi-isometric free groups in $\SL_3(\R)$]{Quasi-isometric free group representations into \(\SL_3(\R)\)}
\author{Le\'on Carvajales \and Pablo Lessa \and Rafael Potrie}
\address{}
\thanks{The authors were partially supported by ANII-FCE Agencia Nacional de Investigaci\'on e Innovaci\'on -
FCE\_3\_2020\_1\_162840, Ministerio de Educaci\'on y Cultura DICYT - FVF\_2023\_436, CSIC Grupos 2022 "Geometr\'ia y Acciones de Grupos" and CSIC I+D "Estructuras topologicas de sistemas parcialmente hiperbólicos y aplicaciones''.}
\address{\newline Le\'on Carvajales \newline Universidad de la Rep\'ublica \newline Facultad de Ciencias Econ\'omicas y de Administraci\'on \newline Instituto de Estad\'istica \newline CNRS IRL IFUMI \newline e-mail: leon.carvajales@fcea.edu.uy \newline \newline
Rafael Potrie \newline Universidad de la Rep\'ublica \newline Facultad de Ciencias \newline CNRS IRL IFUMI
\newline Centro de Matem\'atica \newline e-mail: rpotrie@cmat.edu.uy \newline \newline
Pablo Lessa \newline Universidad de la Rep\'ublica \newline Facultad de Ciencias
\newline Centro de Matem\'atica \newline e-mail: lessa@cmat.edu.uy
}

\begin{document}

\begin{abstract}
We study quasi-isometric representations of finitely generated non-abelian free groups into some higher rank semi-simple Lie groups which are not Anosov, nor approximated by Anosov. 
We show in some cases that these can be perturbed to be non-quasi-isometric, or to have some instability properties with respect to their action on the flag space.
\end{abstract}
 
 \maketitle

\setcounter{tocdepth}{1}
\tableofcontents

\section{Introduction}

\subsection{Main problem}
 
This article attempts to make a contribution towards understanding the following:
\begin{problem}\label{prob1}
Let \(\rho:\Gamma \to \sG\) be a representation from a finitely generated non-abelian free group \(\Gamma\) to a connected semi-simple Lie group \(\sG\).  If \(\rho\) is robustly quasi-isometric, does it follow that \(\rho\) is Anosov?
\end{problem}

We will now recall the necessary definitions.  We say a property of a representation \(\rho:\Gamma \to \sG\) is \textit{robust} if it holds in an open neighborhood of \(\rho\), in the topology of pointwise convergence.
This topology is equivalent to the one induced by evaluating a representation on a given finite generating set.

Quasi-isometric representations are defined as follows.
Fixing \(F \subset \Gamma\) a finite symmetric generating set of \(\Gamma\)\footnote{Throughout, we assume that the rank of $\Gamma$ as a free group is $\geq 2$.}, we recall that the \textit{word length} of an element \(\gamma \in \Gamma\) is defined  by
\[|\gamma| := \min\lbrace n \ge 1: \gamma \in F^n\rbrace,\]
with the convention that \(|1|\) (the word length of the neutral element) is \(0\), and \(F^n := \lbrace f_1\cdots f_n:  f_1,\ldots,f_n \in F\rbrace\).
Fixing any left-invariant Riemannian metric on \(\sG\) and letting \(\dist\) denote the associated distance, a representation \(\rho:\Gamma \to \sG\) is said to be \textit{quasi-isometric} if it is a quasi-isometric embedding of $\Gamma$, that is, there exist constants \(a \ge 1\) and \(b \ge 0\) such that
\[a^{-1}|\gamma^{-1}\eta| - b\le \dist(\rho(\gamma),\rho(\eta)) \le a|\gamma^{-1}\eta| + b,\]
for all \(\gamma,\eta \in \Gamma\). 

Finally, we recall the definition of Anosov representations.
We fix a Cartan subspace \(\a\) of $\sG$, a Weyl chamber \(\ap\subset\a\), and a corresponding Cartan projection \(a:\sG \to \ap\).  
A representation is said to be \textit{Anosov} if there exist constants \(c,d > 0\) and a simple root \(\alpha\) such that
\begin{equation}\label{eqanosov}c |\gamma| - d \le \alpha(a(\rho(\gamma))),\end{equation}
for all \(\gamma \in \Gamma\).
In that case we say that $\rho$ is $\{\alpha\}$-\textit{Anosov}.

Anosov representations were introduced by Labourie \cite{labourie} in his study of the Hitchin component \cite{hitchin}.
They were further generalized to arbitrary word hyperbolic groups by Guichard-Wienhard \cite{GW}. 
If $\sG$ is rank one, they coincide with convex co-compact representations. 
They are nowadays understood as the correct generalization of this notion to higher rank. 
See the surveys of Kassel \cite{kassel} and Wienhard \cite{weinhard} for further information and motivations.
In particular, being Anosov is a robust property \cite{labourie,GW} (see also \cite[Corollary 5.10]{BPS}).
We point out that the original definition of Anosov representations is not the one given in Equation (\ref{eqanosov}), but this is an equivalent one \cite{BPS,GGKW,KLP}. 

From Equation (\ref{eqanosov}) one readily sees that Anosov representations are quasi-isometric. 
While in rank one these two notions coincide, in higher rank being quasi-isometric is not a robust property (see Subsection \ref{subsec: non anosov examples intro} below).

Problem \ref{prob1} was suggested in Bochi-P.-Sambarino \cite[\S 4.4]{BPS} and then formalized in P. \cite[\S 4.3]{potrie}. 
The question also appears in \cite[\S 8]{kassel}. 
While the question makes sense for groups more general than free groups (and was posed in that way in the mentioned references), one needs to be careful about which groups can admit a positive answer.
Indeed, recently Tsouvalas \cite{tsouvalas2} presented examples showing that the answer is negative for general hyperbolic groups, see \S \ref{sss derived barbot intro} for further details.

\subsection{Rank \(1\) groups and products}

For \(\sG\) of rank \(1\), Problem \ref{prob1} is trivial because quasi-isometric representations and Anosov representations coincide.  However, the following question is still interesting:

\begin{problem}\label{prob2}
Suppose that \(\sG\) is a rank \(1\) connected semi-simple Lie group and \(\rho\) is a  robustly discrete and faithful representation of a non-abelian free group into $\sG$. 
Is \(\rho\) necessarily Anosov?
\end{problem}

In the case \(\sG =  \SL_2(\R)\) the answer to Problem \ref{prob2} is affirmative and is an elementary fact, see \S~\ref{ss.SL2} for a proof\begin{footnote}{However for semi-groups the problem is more subtle and studied carefully by Avila-Bochi-Yoccoz \cite{avila-bochi-yoccoz}.}\end{footnote}.  
For \(\sG=  \SL_2(\C)\) the answer is still affirmative and is a result of Sullivan \cite{sullivan}. For other rank \(1\) groups, such as \(\SO_0(1,d)\) with \(d>3\), the problem remains open as far as the authors are aware (c.f. Kapovich \cite[Question 11.14]{kapkleiniansurvey}).

We will give a partial answer to Problem \ref{prob1} for representations into certain products of rank one Lie groups in \S~\ref{ss.prodSL2}:

\begin{theorem}\label{theoproduct}
Suppose that \(\Gamma\) is a finitely generated non-abelian free group and \(\sG = \SL_2(\R)^r \times \SL_2(\C)^s\) for some non-negative integers $r$ and $s$.  Then \(\rho:\Gamma \to \sG\) is robustly faithful if and only if it is Anosov. 
\end{theorem}

We remark that in the context of Theorem \ref{theoproduct} the fact that \(\rho\) is Anosov is equivalent to the existence of a projection onto some factor which is quasi-isometric.
We also observe that Theorem \ref{theoproduct} remains valid when replacing $\Gamma$ by a surface group, see Remark \ref{rem: surface groups in rank one products}.

\subsection{Free groups in \(\SL_3(\R)\)}

For \(\sG = \SL_d(\R)\) one may take the simple roots \(\{\alpha_i\}_{i=1}^{d-1}\) so that they satisfy
\[\alpha_i(a(\rho(\gamma)) = \log s_{i}(\rho(\gamma)) - \log s_{i+1}(\rho(\gamma)).\] \noindent Here \(s_1(\rho(\gamma)) \ge \cdots \ge s_d(\rho(\gamma))\) denote the singular values of \(\rho(\gamma)\) with respect to the standard Euclidean inner product on \(\R^d\) (in particular, $s_1(\rho(\gamma))$ is the standard operator norm of the matrix).
In this case, if the Anosov condition (\ref{eqanosov}) is satisfied for some \(\alpha_i\), it is also satisfied by \(\alpha_{d+1-i}\).  
In particular, Anosov representations into \(\SL_3(\R)\) always satisfy the condition (\ref{eqanosov}) for \(\alpha=\alpha_1\).
On the other hand, in this context the representation \(\rho\) is quasi-isometric if and only if \begin{equation}\label{eq: QI in SLd}
\log s_1(\rho(\gamma))\geq a^{-1}\vert\gamma\vert-b
\end{equation}\noindent for all $\gamma\in\Gamma$ and for some $a\geq 1$ and $b\geq 0$.

\subsubsection{Non-Anosov Examples}\label{subsec: non anosov examples intro}

There exist quasi-isometric representations of the free group into \(\SL_3(\R)\) which are not in the closure of the set of Anosov representations.
There are two types of examples.
The first class was studied by Lahn \cite{lahn} and will be reviewed in \S~\ref{ss.reducible}.
These representations are always reducible.
In \S~\ref{subsec: ex more generators} we prove the following:

\begin{theorem}\label{teo: ejemplonoreducible}
Let $k>2$ be an integer and $\Gamma_k$ be a non abelian free group of rank $k$.
There exists a quasi-isometric representation \(\rho_k:\Gamma_k \to \SL_3(\R)\) whose image is Zariski dense, and such that \(\rho_k\) is not accumulated by Anosov representations.
Furthermore, the representations $\rho_k$ are not robustly quasi-isometric.
\end{theorem}

The examples in Theorem \ref{teo: ejemplonoreducible} are all constructed by restricting a particular quasi-isometric representation $\rho_2:\Gamma_2\to\SL_3(\R)$ of a non abelian free group in two generators to an appropriate finite index subgroup. The representation $\rho_2$ is also not a limit of Anosov representations, but we don't know whether it is robustly quasi-isometric or not. 
Nevertheless, by applying a recent result by Dey-Hurtado \cite{DeyHurfulllimit} we can show that in any neighborhood of \(\rho_2\) there exists a representation $\rho_2':\Gamma_2\to\SL_3(\R)$ whose action on the space of full flags \(\F\) of \(\R^3\) is minimal\footnote{Recall that the space of full flags is the space of pairs \((L,P)\) where \(L\) is a one dimensional subspace of \(\R^3\), \(P\) is a two dimensional subspace, and \(L \subset P\).  
The space of flags is included in, and inherits its topology from, the product \(\RP^2 \times \Gr_2(\R^3)\), where \(\RP^2\) is the real projective plane, and \(\Gr_2(\R^3)\) is the Grassmannian manifold of two dimensional subspaces.}.
This is in constrast with what happens with the starting representation $\rho_2$, which in fact has a proper \textit{limit set} in $\mathcal{F}$.

Problem \ref{prob1} can be compared with the problem of understanding structural stability or robust transitivity of diffeomorphisms (see \cite[\S 4.4]{BPS}).
This is somewhat the viewpoint of Sullivan \cite{sullivan}. 
The basic example $\rho_2$ alluded above, while we are not able to show that it is not robustly quasi-isometric, we can show that the action on the flag space is not structurally stable (in fact, the topology of the limit set changes by perturbation). 
This is also related to Dey-Hurtado \cite[Question 1.1]{DeyHurfulllimit} which asks whether a discrete subgroup of $\SL_3(\R)$ with full limit is necessarily a lattice.
If that question admits a positive answer, it would imply that $\rho_2$ is not robustly faithful and discrete.

Before moving on to the statement of our main result we mention here that Guichard first constructed examples of quasi-isometric representations of a non-abelian free group into $\SL_2(\R)\times\SL_2(\R)$ which are not stable under small deformations.
In fact, they are accumulated by dense representations (see Gu\'eritaud-Guichard-Kassel-Wienhard \cite[Appendix A]{GGKW}).
After that, Tsouvalas \cite{tsouvalas} constructed examples of quasi-isometric representations into $\SL_d(\R)$ which are not limits of Anosov representations, for some specific values of $d$ and some specific word hyperbolic groups. 
In particular, \cite[Proposition 4.2]{tsouvalas} constructs examples of quasi-isometric representations of free groups into $\SL_6(\R)$ which are not accumulated by Anosov representations. 
These examples are strongly irreducible.
The basic example $\rho_2$ improves this result by finding Zariski dense examples in lower dimensions.

\subsubsection{Derived from Barbot representations}\label{sss derived barbot intro}

Our main result is to provide a positive answer to Problem \ref{prob1} for reducible representations into $\SL_3(\R)$.
As being Anosov and quasi-isometric are properties which are preserved under taking duals, we will always assume that our reducible representations preserve a hyperplane $P\subset\R^3$.
Following Lahn \cite{lahn} we will call such a representation a \emph{reducible suspension}.

For every reducible suspension $\rho:\Gamma\to\SL_3(\R)$ we may define
\[\rho_P:\Gamma \to \SL^{\pm}(P): \hspace{0,2cm} \rho_P(\gamma):= \frac{1}{|\det_P(\rho(\gamma))|^{\frac{1}{2}}}\cdot\rho(\gamma)\vert_{P}. \]
\noindent In the above formula, \(P\) is given the volume form inherited from the standard volume form in \(\R^3\), \(|\det_P(g)|\) denotes the Jacobian of a linear map $g$ restricted to \(P\), and $\SL^+(P)=\SL(P)$ (resp. $\SL^-(P)$) denotes the set of linear maps of $P$ whose determinant is equal to $1$ (resp. $-1$).
A reducible suspension will be called \emph{derived from Barbot} if \(\rho_P\) is quasi-isometric. 
This terminology is motivated by the fact that was Barbot \cite{barbot} who first constructed examples of reducible suspensions which are Anosov, these examples are all derived from Barbot. 
However, a general derived from Barbot representation is not necessarily Anosov, but obtained from an Anosov one by a very large deformation. This is reminiscent to the construction of \textit{derived from Anosov} diffeomorphisms in hyperbolic dynamics \cite{mane}.

In \S~\ref{ss.proofpropequiv} we prove the following.

\begin{proposition}\label{p.equiv}
Let $\Gamma$ be a finitely generated non-abelian free group and \(\rho:\Gamma \to \SL_3(\R)\) be a reducible suspension preserving some hyperplane $P$.
Then the following are equivalent:
 \begin{enumerate}
  \item the representation \(\rho\) is quasi-isometric,
  \item the representation \(\rho\) is derived from Barbot,
  \item the representation \(\rho\) is robustly quasi-isometric among representations preserving \(P\).
 \end{enumerate}
\end{proposition}

The above result is  partially proved in \cite{lahn} in a more general setting.
In particular, Lahn proves that a reducible suspension of a quasi-isometric representation is always quasi-isometric \cite[Proposition 3.4]{lahn}.
Moreover, if $\rho$ is faithful and discrete then $\rho_P$ is also faithful and discrete by \cite[Proposition 3.1]{lahn}.
Hence, the only novelty in Proposition \ref{p.equiv} is the fact that $\rho$ being quasi-isometric implies that $\rho_P$ is also quasi-isometric (see Corollary \ref{cor:redsuspensionQIisDFB}), compare with \cite[Proposition 4.2]{lahn}.

The main result of this note is to provide an affirmative answer to Problem \ref{prob1} for reducible suspensions.
By Proposition \ref{p.equiv}, we may restrict our attention to derived from Barbot representations.
Lahn \cite[Theorem 2]{lahn} gives a description of which derived from Barbot representations are Anosov. 
Starting from this result we prove the following.

\begin{theorem}\label{t.main}

Let $\Gamma$ be a finitely generated non-abelian free group and \(\rho:\Gamma \to \SL_3(\R)\) be a derived from Barbot representation.  
Assume moreover that $\rho_P$ preserves the orientation of $P$.
Then \(\rho\) is robustly quasi-isometric if and only if it is Anosov. 
\end{theorem}

We point out here that examples of robustly quasi-isometric representations which fail to be Anosov were found recently by Tsouvalas \cite{tsouvalas2}.
These are representations of free products of uniform lattices in $\mathsf{Sp}(n,1)$ into $\SL_d(\R)$ for large enough $d$. 
Even though these representations admit non trivial deformations, these  deformations are quite special thanks to Corlette's Archimedean Superrigidity \cite{corlette}.
This is the main reason why we have decided to focus on free groups in Problem \ref{prob1}, in order to have enough flexibility to work with. 
We also believe that Problem \ref{prob1} is interesting when replacing $\Gamma$ by a (closed) surface group, as these are also quite flexible (c.f. Remark \ref{rem: surface groups in rank one products}).
In this direction, we mention that similar questions were studied by Danciger-Gu\'eritaud-Kassel-Lee-Marquis \cite{DGKLMconvexcoxeter} for some Coxeter groups embedded into some $\SL_d(\R)$.
In particular, in these cases the interior of the set of representations which arise as a reflection group is shown to coincide with some space of Anosov representations (see \cite[Corollary 1.18]{DGKLMconvexcoxeter}). 
We point out also that the authors found an example of a robustly faithful and discrete representation of a word hyperbolic (Coxeter) group which is not Anosov, see \cite[Remark 1.19]{DGKLMconvexcoxeter}.


\subsection{Organization of the paper}

In Section \ref{sec: products} we prove Theorem \ref{theoproduct}, for which we first address Problem \ref{prob2} for $\sG=\SL_2(\R)$ (Subsection \ref{ss.SL2}).
In Section \ref{sec: generalities QI} we recall generalities about quasi-isometric representations into $\SL_3(\R)$, and prove Proposition \ref{prop: criterion for non robust qi}. This proposition is one of the two main ingredients in the proof of Theorem \ref{t.main}, as it provides a sufficient condition for a representation not to be robustly quasi-isometric.
Theorem \ref{teo: ejemplonoreducible} is proven in Section \ref{sec: examples}, more concretely in \S~\ref{subsec: ex more generators}, after constructing the basic example $\rho_2$ in Subsections \ref{ss.ZD} and \ref{ss.minimal}.
Previous to that, in \S~\ref{ss.reducible} we recall recent work by Lahn \cite{lahn} discussing reducible examples of quasi-isometric representations which are not limits of Anosov representations. 
In Section \ref{sec. derived from barbot} we prove Theorem \ref{t.main}.
We first prove Proposition \ref{p.equiv} in Subsection \ref{ss.proofpropequiv}.
In Subsection \ref{ss.multiplicity} we prove Proposition \ref{prop-invariantplanejordan}, which is the second main ingredient in the proof of Theorem \ref{t.main}, as it shows that a non-Anosov reducible quasi-isometric representation can be perturbed so that a specific element has a repeated eigenvalue.
We then observe in Subsection \ref{subsec: proof of main thm} that this proposition together with Proposition \ref{prop: criterion for non robust qi} readily imply Theorem \ref{t.main}.

\subsection{Acknowledgements}
The authors are grateful to Subhadip Dey, Sebasti\'an Hurtado, Fanny Kassel, Max Lahn and Konstantinos Tsouvalas for interesting discussions and comments. Besides interesting discussions, Andr\'es Sambarino suggested to look at products of rank 1 groups as a first problem.

\section{Products of $\SL_2$'s}\label{sec: products}

\subsection{Robust faithfulness in  \(\SL_2(\R)\)}\label{ss.SL2}

Here we show that robustly faithful representations of a free group into $\SL_2(\R)$ are Anosov (note that this does not follow from Sullivan's Theorem \cite{sullivan}).
This will be used in the proof of Theorem \ref{theoproduct}.

We begin with a preparatory lemma.

\begin{lemma}\label{lem: commutator is submersion}
Let $g_0$ and $h_0$ be non commuting hyperbolic elements in $\SL_2(\R)$.
Then the commutator $$[\cdot,\cdot]:\SL_2(\R)\times\SL_2(\R)\to\SL_2(\R); [g,h]:=ghg^{-1}h^{-1}$$ \noindent is a submersion at $(g_0,h_0)$.
\end{lemma}

\begin{proof}
This is a direct computation, we include a proof for completeness.

Let $\liesl$ be the Lie algebra of $\SL_2(\R)$, that is, the vector space of traceless $2\times 2$ real matrices. 
For $g\in\SL_2(\R)$, the tangent space $T_g\SL_2(\R)$ is $g\cdot\liesl$. 

Let $X,Y\in\liesl$.
The derivative of $[\cdot,\cdot]$ at $(g_0,h_0)$ evaluated on the pair $(g_0\cdot X,h_0\cdot Y)$ is given by $$[g_0,h_0]\cdot \left(\text{Ad}_{h_0g_0h_0^{-1}}(X)+\text{Ad}_{h_0g_0}(Y)-\text{Ad}_{h_0g_0}(X)-\text{Ad}_{h_0}(Y)\right),$$ \noindent where $\text{Ad}:\SL_2(\R)\to\text{Aut}(\liesl)$ is the adjoint representation. 
In particular, it suffices to show that the linear map $$(X,Y)\mapsto \text{Ad}_{g_0}\left(\text{Ad}_{h_0^{-1}}(X)-X\right)+\left(\text{Ad}_{g_0}(Y)-Y\right)$$ \noindent is surjective. 

Without loss of generality we may suppose that $g_0$ is diagonal, with eigenvalues $\mu\neq \pm 1$ and $\mu^{-1}$. 
Then $\text{Ad}_{g_0}$ is diagonalizable on the basis $\left\lbrace\left(\begin{smallmatrix}
1 & 0 \\
0 & -1
\end{smallmatrix}\right),\left(\begin{smallmatrix}
0 & 1 \\
0 & 0
\end{smallmatrix}\right),\left(\begin{smallmatrix}
0 & 0 \\
1 & 0
\end{smallmatrix}\right)\right\rbrace$, with respective eigenvalues $1$, $\mu^2$ and $\mu^{-2}$.
The image of $Y\mapsto \text{Ad}_{g_0}(Y)-Y$ is therefore $$W_{g_0}:=\text{span}\left\lbrace\left(\begin{matrix}
0 & 1 \\
0 & 0
\end{matrix}\right),\left(\begin{matrix}
0 & 0 \\
1 & 0\end{matrix}\right)\right\rbrace.$$
\noindent Moreover, as $h_0$ is hyperbolic but not diagonal (because it does not commute with $g_0$), it follows easily that $W_{g_0}$ is different from the image of $$X\mapsto \text{Ad}_{g_0}\left(\text{Ad}_{h_0^{-1}}(X)-X\right),$$ \noindent which is again two-dimensional. 
This completes the proof.
\end{proof}

\begin{proposition}\label{prop:robustfaithfulSL2}
Let \(\Gamma\) be a finitely generated non-abelian free group and \(\rho: \Gamma \to \SL_2(\R)\) a robustly faithful representation. Then \(\rho\) is Anosov (equivalently, quasi-isometric). 
\end{proposition}

\begin{proof}

We first observe that $\rho$ must be discrete. 
Otherwise, the Euclidean closure \(\overline{\rho(\Gamma)}\) of $\rho(\Gamma)$ would be either $\SL_2(\R)$ or contained in a (virtually) solvable subgroup of $\SL_2(\R)$.
The latter case is ruled out, as it would imply that $\rho$ has a non trivial kernel.
Hence \(\overline{\rho(\Gamma)} = \SL_2(\R)\) and there is an element  $\gamma\in\Gamma\setminus\{1\}$ so that trace $\mathrm{tr}(\rho(\gamma))$ belongs to the interval $(-2,2)$. Since the trace function \(\rho' \mapsto \mathrm{tr}(\rho'(\gamma))\) is a non-constant polynomial in the entries of the image of a free generating set of \(\Gamma\), it must take a value of the form \(2 \cos(\frac{k \pi}{n})\) for some integers $k$ and $n$ and arbitrarily small perturbations \(\rho'\) of \(\rho\). Thus we would get a small perturbation of $\rho$ which is not faithful contradicting our assumption. 

Now that we know that \(\rho(\Gamma)\) is discrete, observe that the induced representation $\Gamma\to\PSL_2(\R)$ is also discrete. 
It is moreover robustly faithful and, to finish the proof, it suffices to show that this induced representation is quasi-isometric. 
By abuse of notations, we still denote this induced representation by $\rho$.

Consider the hyperbolic surface $\HH^2/\rho(\Gamma)$, which is geometrically finite as $\Gamma$ is finitely generated (see e.g. \cite[Theorem 4.6.1]{skatok}). 
We will show that every element in $\rho(\Gamma)$ is hyperbolic, and this will finish the proof (see e.g. \cite[Theorems 4.8 \& 4.13 and Corollary 4.17]{dalbogeodesichorocyclic}).


Suppose by contradiction that there is some parabolic element corresponding to some cusp in $\HH^2/\rho(\Gamma)$. There are two cases, namely, either $\HH^2/\rho(\Gamma)$ has genus $g=0$, or $g>0$.

If $g=0$, we may suppose that the cusp corresponds to a free generator of $\Gamma$. 
It can then be perturbed to a finite order element in $\PSL_2(\R)$, contradicting the fact that $\rho$ is robustly faithful. On the other hand, if $g>0$ then the cusp corresponds to an element $\gamma$ of $\Gamma$ so that there is a free generating set $$\{a_1,b_1,\dots,a_g,b_g,c_1,\dots,c_n\}\subset\Gamma$$ \noindent such that $$\gamma=[a_1,b_1]\dots[a_g,b_g]c_1\dots c_n.$$
\noindent We can assume that both $\rho(a_1)$ and $\rho(b_1)$ are hyperbolic, otherwise we conclude as in the previous case. 
Since $\rho$ is faithful, by Lemma \ref{lem: commutator is submersion} we may perturb $\rho$ only perturbing the image of $a_1$ and $b_1$, in such a way that $\rho(\gamma)$ becomes a finite order element of $\PSL_2(\R)$.
This finishes the proof.

\end{proof}

\subsection{Proof of Theorem \ref{theoproduct}}\label{ss.prodSL2}

Theorem \ref{theoproduct} follows from combining Proposition \ref{prop:robustfaithfulSL2}, Sullivan's Theorem \cite{sullivan} and the following general lemma.

\begin{lemma}\label{lem: products}
Suppose that \(\Gamma\) is a finitely generated non-abelian free group and 
\(\sG = \prod_{i=1}^{r} \sG_i\) for some positive integer $r$ and some Lie groups $\sG_1,\dots,\sG_r$.  Then if \(\rho=(\rho_1,\dots,\rho_{r}):\Gamma \to \sG\) is robustly faithful, there is some $i=1,\dots,r$ such that $\rho_i$ is robustly faithful.
\end{lemma}

\begin{proof}
	By an inductive procedure, it suffices to show that either $\rho_{1}$ or $(\rho_2,\dots,\rho_{r})$ is robustly faithful. 
	Suppose by contradiction that both $\rho_{1}$ and $(\rho_2,\dots,\rho_{r})$ are not. 
	We may then find arbitrarily small perturbations $\rho_{1}'$ and $(\rho_2',\dots,\rho_{r}')$ of $\rho_{1}$ and $(\rho_2,\dots,\rho_{r})$ respectively, and elements $\gamma_1,\gamma_2\neq 1$ in $\Gamma$ so that $$\rho_1'(\gamma_1)=1 \text{ and } \rho_i'(\gamma_2)=1$$ \noindent for all $i=2,\dots,r$.
	
	The commutator $\gamma:=[\gamma_1,\gamma_2]$ belongs to the kernel of $\rho':=(\rho_1',\dots,\rho_{r}')$, which is an arbitrarily small perturbation of $\rho$. Hence $\gamma=1$ and there is some pair of integers such that $\gamma_1^{n}=\gamma_2^m$. It follows that $\gamma_2^{m}$ belongs to the kernel of $\rho'$, thus finding the desired contradiction.
	
\end{proof}

\begin{rem}\label{rem: surface groups in rank one products}
We observe here that Theorem \ref{theoproduct} remains valid when replacing $\Gamma$ by the fundamental group of a closed orientable surface. 
Indeed, Proposition \ref{prop:robustfaithfulSL2} remains true in that case thanks to Funar-Wolff \cite[Theorem 1.1]{wolff} (see also \cite{deblois}).
The corresponding statement for $\SL_2(\C)$ is also true thanks to Sullivan \cite{sullivan}.
This shows the claim, as Lemma \ref{lem: products} is completely general.
\end{rem}

\section{Perturbed not quasi-isometric representations}\label{sec: generalities QI}

In this section we prove a crucial result (Proposition \ref{prop: criterion for non robust qi}) which gives a sufficient condition ensuring that a representation of a free group into $\SL_3(\R)$ is not robustly quasi-isometric. 
The proof is given in Subsection \ref{ss.criterionnonQI}.
We begin by recalling general properties of quasi-isometric representations in Subsection \ref{subsec: propsQIreps}.

\subsection{General properties of quasi-isometric representations}\label{subsec: propsQIreps}

Let $\rho$ be a quasi-isometric representation of a finitely generated free group $\Gamma$ into $\GL_d(\R)$. 
It is not hard to see that every conjugate of $\rho$ is also quasi-isometric.
Moreover, $\rho$ is necessarily discrete and faithful, as $\Gamma$ is torsion free.

We record two other general properties for future use.

\begin{rem}\label{rem: finite index}
Suppose that $\Gamma_0$ is a finite index subgroup of $\Gamma$. 
Then $\Gamma_0$ is finitely generated and quasi-isometric to $\Gamma$. 
In particular, a representation $\rho:\Gamma\to\GL_d(\R)$ is quasi-isometric if and only if its restriction to $\Gamma_0$ is quasi-isometric.
\end{rem}

The most important property that we will need is the following.

\begin{rem}\label{rem: unipotents} 
Recall that an element $g\neq 1$ of \(\GL_d(\R)\) is said to be \emph{unipotent} if all its eigenvalues are equal to $1$. 
It is necessarily conjugate to an upper triangular matrix.
Moreover, for a given norm $\Vert\cdot\Vert$ on $\R^d$ the powers \(\Vert g^n\Vert\) grow in a polynomial way as $n\to\infty$.
In particular, the image of a quasi-isometric representation into $\GL_d(\R)$ does not contain unipotent elements. 
\end{rem}

\subsection{A criterion for non quasi-isometry}\label{ss.criterionnonQI}

For a matrix $g\in\SL_3(\R)$ we let $$\lambda_u(g)\geq\lambda_c(g)\geq\lambda_s(g) $$ \noindent be the moduli of the eigenvalues of $g$.
Recall that $g$ is said to be \textit{loxodromic} if the inequalities above are all strict.
In this case, the eigenvalues of $g$ are real, and denoted respectively bu $\mu_u(g),\mu_c(g)$ and $\mu_s(g)$. 
The corresponding eigenspaces are one dimensional and denoted by $\Eu(g),\Ec(g)$ and $\Es(g)$. 
We also let $$\Ecu(g):=\Eu(g)\oplus\Ec(g) \text{ and } \Ecs(g):=\Es(g)\oplus\Ec(g),$$ \noindent which are respectively the attracting and repelling hyperplane of $g$ acting on the Grassmannian $\Gr_2(\R^3)$.

We have the following general lemma.

\begin{lemma}\label{lem: power map is open}
Let $g_0\in\SL_3(\RR)$ be a loxodromic element and $n$ be a non-zero integer. 
Then the map $$\SL_3(\RR)\to \SL_3(\RR): g\mapsto g^n$$ \noindent sends every small enough neighborhood $U$ of $g_0$ to a neighborhood $U^{(n)}$ of $g_0^n$.
\end{lemma}

\begin{proof}
Indeed, every power of a loxodromic element is again loxodromic. 
Moreover, a neighborhood of a loxodromic element is parametrized by the choice of three lines $\Eu,\Ec$ and $\Es$ in general position, plus the choice of three non zero real numbers $\mu_u,\mu_c$ and $\mu_s$ of different moduli, and whose product is equal to one.
As the eigenlines of a power of $g$ coincide with eigelines of $g$, and the map $$\R\setminus\{0\}\to \R\setminus\{0\}: \mu\mapsto \mu^n$$ \noindent is open, the proof is complete. 
\end{proof}

\begin{prop}\label{prop: criterion for non robust qi}
Let $\Gamma$ be a non-abelian free group and $\rho:\Gamma\to \SL_3(\R)$ be a representation.
Suppose that $\Gamma$ admits a free generating set containing two elements $a$ and $b$ with the following properties:
\begin{enumerate}
\item there exist non zero integers $m$ and $n$ so that the word $\omega:=a^mb^n$ satisfies that for some non zero integer $q$ and some $\mu\neq \pm 1$ the subspace $P_0:=\ker(\rho(\omega^q)-\mu)$ is two-dimensional.
Moreover, the subspace $L_0:=\ker(\rho(\omega^q)-\mu^{-2})$ is one-dimensional. \item The matrices $\rho(a)$ and $\rho(b)$ are loxodromic and $L_0$ is contained either in the attracting plane $\Ecu(\rho(a))$ of $\rho(a)$, or in the repelling one $\Ecs(\rho(a))$.
\end{enumerate}
Then there exists a continuous path $\rho_t:\Gamma\to\SL_3(\R)$ with $\rho_0=\rho$ and a sequence $t_k\to 0$ such that for every $k$ the element $$\rho_{t_k}([\omega^q,a^{p_k}\omega^qa^{-p_k}])$$ \noindent is unipotent for some non zero integer $p_k$. 
\end{prop}

\begin{proof}
Without loss of generality we assume that $L_0$ belongs to $\Ecu(\rho(a))$.

Let us first observe that by taking an arbitrarily small perturbation of $\rho$ if needed we may assume, in addition to the current hypotheses, that $\Ec(\rho(a))$ is different from $L_0$ and $\Es(\rho(a))$ is not contained in $P_0$.
Indeed, there exist arbitrarily small neighborhoods $U$ of $\rho(b)$ so that $U^{(n)}$ is a neighborhood of $\rho(b^{n})$ (Lemma \ref{lem: power map is open}).
We may take an arbitrarily small perturbation $\rho'(a)$ of $\rho(a)$ so that $\Ec(\rho'(a))$ is different from $L_0$, $\Es(\rho(a))$ does not belong to $P_0$, $L_0$ belongs to $\Ecu(\rho'(a))$, and $$\rho'(a)^{-m}\rho(a^{m})\rho(b^{n}) \in U^{(n)}.$$
\noindent In particular, there is some $\rho'(b)\in U$ so that $$\rho'(a)^{-m}\rho(a^{m})\rho(b^{n})=\rho'(b)^{n}.$$
\noindent By keeping all other generators of $\Gamma$ fixed, this defines a perturbation $\rho'$ of $\rho$ for which $\rho'(\omega)=\rho(\omega)$, and $\rho'$ is still in the hypotheses of the proposition, with the extra claimed genericity assumptions on $\Ec(\rho'(a))$ and $\Es(\rho'(a))$.
To lighten notations, we assume the representation $\rho$ that we started with also satisfies these conditions.

Now fix any continuous path $t\mapsto\rho_t(a)$ passing through $\rho(a)$ at $t=0$, which is constructed only deforming eigenline $\Eu(\rho(a))$ in such a way that $t\mapsto \Eu(\rho_t(a))$ intersects transversely $\Ecu(\rho(a))$ precisely at $t=0$. 
As above, we may find a continuous path $t\mapsto \rho_t(b)$ passing through $\rho(b)$ at $t=0$ and such that $\rho_t(\omega)=\rho(\omega)$ for all $t$.

Fix a unit vector $v\in L_0$ and let $t\mapsto \theta_t\in(\RR^3)^*\setminus\{0\}$ be a continuous path so that $$\ker\theta_t=\Ecu(\rho_t(a))$$ \noindent for all $t$.
As $\Ec(\rho(a))=\Ec(\rho_t(a))\neq L_0$ and $L_0\in\Ecu(\rho(a))$, by our transversality assumption we may find arbitrarily small $t_-<0<t_+$ so that 
$$\theta_{t_-}(v)\cdot \theta_{t_+}(v)<0.$$ \noindent Without loss of generality we may assume \begin{equation}\label{eq: criterion for non robust qi}\theta_{t_-}(v)<0<\theta_{t_+}(v).\end{equation}

On the other hand, for every positive integer $p$ we may take a continuous path $t\mapsto\theta_{t,p}\in(\RR^3)^*\setminus\{0\}$ so that $$\ker\theta_{t,p}=\rho_t(a^p)\cdot P_0$$ \noindent holds for all $t$.
Note that, as $\Es(\rho(a))=\Es(\rho_t(a))\not\subset P_0$, for a given $t$ we have $$\rho_t(a^p)\cdot P_0\to\Ecu(\rho_t(a)),$$ \noindent as $p\to\infty$.
In particular, we may assume that $\theta_{t,p}$ is chosen in such a way that for every $t$ one has $$\theta_{t,p}\to\theta_t$$ \noindent as $p\to\infty$.
By Equation (\ref{eq: criterion for non robust qi}) we may find a large enough $p$ such that $$\theta_{t_-,p}(v)<0<\theta_{t_+,p}(v).$$
\noindent There exists then some $t_-<t_0<t_+$ such that $\theta_{t_0,p}(v)=0$ or, in other words, $$L_0\subset \rho_{t_0}(a^{p})\cdot P_0.$$
In particular, $\rho_{t_0}(\omega^q)=\rho(\omega^q)$ preserves the hyperplane $\rho_{t_0}(a^{p})\cdot P_0$, because it preserves $L_0$ and $P_0\cap\rho_{t_0}(a^{p})\cdot P_0$.

On the other hand, the element $\rho_{t_0}(a^p\omega^q a^{-p})$ also preserves $\rho_{t_0}(a^{p})\cdot P_0$ and furthermore acts as a scalar multiple of the identity on it.
Hence, the commutator $\rho_{t_0}([\omega^q,a^p\omega^q a^{-p}])$ acts as the identity on it. 
As $ \rho_{t_0}([\omega^q,a^p\omega^q a^{-p}]) $ has determinant equal to $1$, it is a unipotent element.
\end{proof}

\section{Examples of quasi-isometric free groups in  \(\SL_3(\R)\)}\label{sec: examples}

We now discuss examples of quasi-isometric representations into $\SL_3(\R)$. 
An important class is given by Anosov representations, for which we refer the reader to the surveys \cite{kassel,weinhard}.
However, in this paper we are focusing on examples which are not Anosov, and not even limits of Anosov representations.
We also focus on free groups, but many of the previously known results that we will quote here apply for more general groups.

\subsection{Reducible examples}\label{ss.reducible} 

In this subsection we briefly review a special case of recent results by Lahn \cite{lahn}.
Let $\rho:\Gamma\to\SL_3(\R)$ be a reducible suspension.
Recall that by definition this means that $\rho$ preserves some hyperplane $P$. 
Up to conjugacy, we may write \begin{equation}\label{eq. reducible suspension}
\rho(\gamma)=\left( \begin{matrix}
e^{-\frac{1}{2}\varphi(\gamma)}\rho_P(\gamma) & \kappa(\gamma)\\
0 & \pm e^{\varphi(\gamma)}
\end{matrix} \right),\end{equation} \noindent for some morphisms $\rho_P:\Gamma\to\SL_2^\pm(\RR)$ and $\varphi:\Gamma\to\RR$, and some $\kappa:\Gamma\to\RR^2$.
We recall that $\rho$ is said to be \textit{derived from Barbot} if $\rho_P$ is quasi-isometric.

Lahn \cite[Proposition 3.4]{lahn} proves that if $\rho$ is derived from Barbot, then it is necessarily quasi-isometric (see also Lemma \ref{lem:DFBisQI} below).
Moreover, Lahn characterizes which reducible suspensions are Anosov:

\begin{teo}[Lahn {\cite[Theorem 2]{lahn}}]\label{teo: lahn in examples}
Let $\rho:\Gamma\to\SL_3(\R)$ be a reducible suspension.
Then $\rho$ is Anosov if and only if $$\displaystyle\inf_{\gamma: \varphi(\gamma)\neq 0}\frac{\log(\lambda_u(\rho_P(\gamma)))}{\vert\varphi(\gamma)\vert}>\frac{3}{2},$$
\noindent where $\lambda_u(\rho_P(\gamma))$ denotes the largest modulus among the eigenvalues of $\rho_P(\gamma)$.

\end{teo}

Building on Theorem \ref{teo: lahn in examples} it is easy to construct examples of quasi-isometric reducible suspensions which are not accumulated by Anosov representations.

\subsection{Constructing free groups}\label{ss.ZD}  

We now turn to non reducible examples, proving Theorem \ref{teo: ejemplonoreducible}.
The construction involves several intermediate elementary steps.

\subsubsection{A modified ping-pong criterion} 

We start by giving a criterion which allows us to prove that a subgroup of $\SL_3(\R)$ is freely generated and quasi-isometrically embedded.
A similar variant of the ping-pong lemma may be found in Witte Morris \cite[Lemma 4.9.6]{witte-morris}.

We say a finite symmetric subset \(F \subset \SL_3(\R)\) satisfies condition \(\ast\) if there exists \(c > 1\), a one dimensional subspace \(L_0\), and for each \(g \in F\) a subset \(C_g \subset \RP^2\) such that 
\begin{enumerate}
 \item \(C_g \cap C_h = \emptyset\) if \(g \neq h\),
 \item \(L_0 \notin \bigcup\limits_{g \in F}C_g\),
 \item \(g\cdot L_0 \in C_g\) for all \(g \in F\),
 \item \(g\cdot C_h \subset C_g\) whenever \(g \neq h^{-1}\), and
 \item \(\|g\vert_L\| \ge c\) for all \(L \in \bigcup\limits_{h \neq g^{-1}}C_h\).
\end{enumerate}

\begin{lemma}[Ping-pong variant]\label{Lem: pingpong}
 If \(F \subset \SL_3(\R)\) is finite, symmetric, and satisfies condition \(\ast\) then \(F\) freely generates a free subgroup which is quasi-isometrically embedded in \(\SL_3(\R)\).
\end{lemma}
\begin{proof}
 Let \(n \ge 1\) and \(g_1,\ldots, g_n \in F\) be such that \(g_k \neq g_{k+1}^{-1}\) for \(k = 1,\ldots, n-1\).
 
We first prove that \(F\) freely generates a free group.  
To see this, observe that \(g_n \cdot L_0 \in C_{g_n}\), and by induction it follows that \((g_1\cdots g_n) \cdot L_0 \in C_{g_1}\).  
In particular this implies \((g_1\cdots g_n)\cdot L_0 \neq L_0\) so that \(g_1\cdots g_n \neq 1\) as required.

We now show that the group generated by \(F\) is quasi-isometrically embedded.   
For this purpose we observe that for \(k = 1,\ldots, n-1\), since \((g_{k+1}\cdots g_n)\cdot L_0 \in C_{g_{k+1}}\) and \(g_{k} \neq g_{k+1}^{-1}\) we have \(\|g_k\vert_{(g_{k+1}\cdots g_n)\cdot L_0}\| \ge c\).   
This implies
\[s_1(g_1\cdots g_n) = \max\limits_{L \in \RP^2}\|(g_1\cdots g_n)\vert_{ L}\| \ge \|(g_1\cdots g_n)\vert_{ L_0}\| \ge c^{n-1},\]
which concludes the proof since \(c > 1\). 
\end{proof}

\subsubsection{Examples verifying condition \(\ast\)}

We now produce a class of examples satisfying condition \(\ast\).
For this purpose we consider elements \(g \in \SL_3(\R)\) such that for some \(\mu > 1\) and \(m \in \N\) one has that \(P_0: = \ker(g^m - \mu)\) is two-dimensional and \(L_0:=\ker(g^m - \mu^{-2}) \) is one-dimensional.
For such an element and for a given finite \(g\)-invariant set \(X \subset \RP^2\) whose elements are contained in \(P_0\), we say that a set \(A \subset \RP^2\) is a \emph{prepared neighborhood of \(X\)} if \(X\) is contained in the interior of \(A\), \(g(A) = A\), and \(A = \bigcup\limits_{x \in A}x \oplus L_0\).
One has the following.

\begin{proposition}\label{beprepared}
Let $g\in\SL_3(\R)$ be as above and let \(\mathcal{A}=\mathcal{A}_X\) be the set of prepared neighborhoods of a given finite subset \(X\subset\RP^2\).  
Then \(A_1 \cap \cdots \cap A_n \in \mathcal{A}\) for all \(A_1,\ldots, A_n \in \mathcal{A}\) and \(\bigcap\limits_{A \in \mathcal{A}}A = \bigcup\limits_{x \in X}x \oplus L_0\).
\end{proposition}

\begin{lemma}\label{fabricaqi}
 Let \(f,g \in \SL_3(\R)\) be such that:
 \begin{enumerate}
 \item there exist one-dimensional eigenspaces \(E^u, E^c, E^s \in \RP^2\) of \(f\) with respective eigenvalues of moduli \(\lambda_u > \lambda_c > \lambda_s\).
 \item There exist a positive integer \(m\) and \(\mu > 1\) such that \(P_0: = \ker(g^m-\mu)\) is two-dimensional and \(L_0: = \ker(g^m - \mu^{-2}) \) is one-dimensional.
  \item The above subspaces satisfy \(E^u,E^c, E^s \not\subset P_0\), \(L_0 \not\subset E^{cu} := E^u \oplus E^c\), and \(L_0 \not\subset E^{cs} := E^s \oplus E^c\).
  \item For all \(0 \le k < m\) one has \(g^k((E^u \oplus L_0) \cap P_0) \not\subset E^{cs} \cup E^{cu}\), and \(g^k((E^s \oplus L_0) \cap P_0) \not\subset E^{cs} \cup E^{cu}\).
 \end{enumerate}
 Then there exists a positive integer \(n_0\) such that \(F_n := \lbrace f^n,f^{-n},g^n,g^{-n}\rbrace\) satisfies condition \(\ast\) for all \(n > n_0\).
\end{lemma}
\begin{proof}
Let \(\angle(V,W)\) denote the minimum angle between two subspaces \(V,W \subset \R^3\) when either \(V\) or \(W\) is one dimensional.  
For $r>0$ and a subset \(X \subset \RP^2\) we denote by \(B(X,r)\) the set of \(L \in \RP^2\) with \(\angle(L',L) < r\) for some \(L' \in X\).

Define
\[X := \bigcup\limits_{k = 0}^{m-1} g^k\left((E^u \oplus L_0) \cap P_0\right) \cup g^k\left((E^s \oplus L_0) \cap P_0\right)\subset\RP^2,\]
\noindent which by hypothesis is a finite set of lines contained in \(P_0\) that is disjoint from \(E^{cs} \cup E^{cu}\).

We consider the family $\mathcal{A}=\mathcal{A}_X$ of prepared neighborhoods of $X$ (with respect to $g$).
By Proposition \ref{beprepared} we may choose \(A\in\mathcal{A}\) such that for some \(\eps_0 > 0\) the set \(A \cap P_0\) is disjoint from \(B(E^{cs}\cup E^{cu},\eps_0)\).

For \(0 < \eps < \eps_0\) which will be specified later we define
\[C_g := A \cap B(P_0,\eps), \text{   } C_{g^{-1}} := A \cap B(L_0,\eps)\] \noindent and
\[C_f := B(E^{u},\eps), \text{   } C_{f^{-1}} := B(E^{s},\eps).\]
\noindent By Hypothesis (3), if \(\eps > 0\) is small enough then the four sets above are pairwise disjoint.
Moreover as \(A\) is \(g\)-invariant, by Hypothesis (2) we have $$g^{n}\cdot(A \setminus C_{g^{-1}}) \subset C_g \text{ and } g^{-n}\cdot(A \setminus C_{g}) \subset C_{g^{-1}}$$ \noindent for all \(n\) large enough. 

Since \(A\) is a prepared neighborhood it contains a neighborhood of each point in \(E^s \oplus L_0\setminus\{L_0\}\), and similarly for \(E^{u} \oplus L_0\).  Therefore, by picking \(\eps > 0\) smaller if necessary, we may assume
\[C_f \cup C_{f^{-1}} \subset A.\]

Let \(F: = \lbrace f,f^{-1},g,g^{-1}\rbrace\).
By Hypothesis (2), we obtain \(g^{n}\cdot(C_h) \subset C_g\) and \[\inf\limits_{L \in C_h}\|g^n\vert_{L}\| > 2\] \noindent for all \(h \neq g^{-1}\) and \(n\) large enough.
Similarly, \(g^{-n}\cdot(C_h) \subset C_{g^{-1}}\) and \[\inf\limits_{L \in C_h}\|g^n\vert_{L}\| > 2\] for all \(h \neq g\) and \(n\) large enough.

On the other hand, since \(B(E^{cs},\eps_0) \cap (A \cap P_0) = \emptyset\), we may choose \(\eps > 0\) smaller if needed so that \(C_g \cap B(E^{cs},\eps) = \emptyset\).   By Hypothesis (3), we may also assume \[C_f\cap B(E^{cs},\eps) = \emptyset \text{ and }  C_{g^{-1}} \cap B(E^{cs},\eps) = \emptyset.\] \noindent Hence, for all \(n\) large enough and \(h \in\{ g,g^{-1},f\}\) we have \(f^n\cdot (C_h) \subset C_f\) and \[\inf\limits_{L \in C_h}\|f^n\vert_{L}\| > 2.\]
\noindent Similarly, one has \(f^{-n}\cdot (C_h) \subset C_{f^{-1}}\) and \[\inf\limits_{L \in C_h}\|f^{-n}\vert_{L}\| > 2\] \noindent for \(h \in\{ f^{-1},g,g^{-1}\}\) and all \(n\) large enough.

To conclude property \(\ast\) is suffices to choose \(L_0 \in A \cap \partial C_{g^{-1}}\) and set \(c := 2\).
\end{proof}

\subsection{The basic example in two generators}\label{ss.minimal}
 
From the general machinery of the previous section we can produce a basic example of a Zariski dense quasi-isometric representation of a free group in two generators into $\SL_3(\R)$, which is not a limit of Anosov representations.
We also show that arbitrary small perturbations of this example act minimally on the space of full flags $\mathcal{F}$ of $\R^3$.
 
Indeed, let \[g = \begin{pmatrix}2 & -2 & 0\\ 2 & 2 & 0\\ 0 & 0 & \frac{1}{8}\end{pmatrix}\] \noindent and \(f\) be loxodromic chosen so that the hypotheses of Lemma \ref{fabricaqi} are satisfied.
By Lemmas \ref{Lem: pingpong} and \ref{fabricaqi} we may take an odd positive integer \(n\) such that \(f^n, g^n\) freely generate a quasi-isometrically embedded free group \(\Gamma\).
 
Let \(\rho:\Gamma \to \SL_3(\R)\) be the inclusion.  
We first claim that \(\rho\) is not accumulated by Anosov representations.
To see this observe that, since \(n\) is odd, the only real eigenvalue of \(g^n\) is \(\frac{1}{8^n}\).  In particular if \(\rho':\Gamma \to \SL_3(\R)\) is sufficiently close to the inclusion then \(h = \rho'(g^n)\) also has a unique real eigenvalue and therefore \(h^k\) has two complex eigenvalues with equal modulus for all \(k\). This implies that \(\rho'\) is not Anosov.

We will now prove that \(f\) can be chosen so that \(\rho(\Gamma)\) is Zariski dense. 
For this, we apply a well-known criterion asserting that a subgroup of $\SL_d(\RR)$ is Zariski dense if and only if its adjoint action on the Lie algebra of $\SL_d(\R)$ leaves no subspace invariant. 
This is quite direct in our case, as the adjoint action of a loxodromic element $f\in\SL_d(\R)$ is diagonalizable thus, an invariant subspace of the adjoint action must be a direct sum of eigenspaces of the adjoint matrix $\text{Ad}_f$. 
Therefore, in our example we may consider $f$ and $g$ so that $\text{Ad}_f$ and $\text{Ad}_{gfg^{-1}}$ are diagonalizable in bases which are in general position, to conclude that the adjoint actions of $f$ and $gfg^{-1}$ do not share any invariant subspace. 
This shows that the inclusion $\rho:\Gamma\to\SL_3(\R)$ has Zariski dense image.

Finally, we show that \(\rho\) may be perturbed to act minimally on the space of full flags \(\F\). 
Indeed, let $P_0:=\text{span}\{e_1,e_2\}$ and $L_0:=\text{span}\{e_3\}$, which are the invariant subspaces of $g$.
Define the representation $\rho'$ by letting \(\rho'(f): = \rho(f)\) and \(\rho'(g)\) preserve \(L_0\) and \(P_0\), but in such a way that \(\rho'(g^n)\vert_{ P_0}\) is not a scalar multiple of the identity for any integer \(n \neq 0\).
By Dey-Hurtado \cite[Proposition 5.4]{DeyHurfulllimit}, $\rho'$ induces a minimal action on $\mathcal{F}$.

\subsection{Proof of Theorem \ref{teo: ejemplonoreducible}}\label{subsec: ex more generators}

We now build on the previous example to prove Theorem \ref{teo: ejemplonoreducible}.
We will need the following abstract lemma.

\begin{lem}\label{lem: free generators of subgroups}
Let $\Gamma$ be the non-abelian free group in the generators $a$ and $b$. 
Then for every $k>2$ there exists a finite index subgroup $\Gamma_k\subset\Gamma$ and a free generating set $F_k=\{c_1,\dots,c_k\}$ of $\Gamma_k$ such that $$c_1=a, \text{ } c_2=b^2, \text{ and } c_3=ba^pb^{-1},$$ for some positive integer $p$.
\end{lem}

\begin{proof}
We identify $\Gamma$ with the fundamental group $\pi_1(X,x_0)$ of a wedge sum $X$ of two oriented circles, glued about the base-point $x_0$, and labelled with letters $a$ and $b$ respectively. 
One may now construct a $(k-1)$-sheeted regular cover $\Pi:\widetilde{X}\to  X$ together with a lifted base-point $\widetilde{x}_0\in \Pi^{-1}(x_0)$ in such a way that $\pi_1(\widetilde{X},\widetilde{x}_0)$ is freely generated by $c_1,\dots, c_k$ and for some $p\geq 1$ one has $$\Pi_*(c_1)=a, \text{ } \Pi_*(c_2)=b^2, \text{ and } \Pi_*(c_3)=ba^pb^{-1},$$ \noindent see Figure \ref{fig: cover}.
This finishes the proof. 
\end{proof}

\begin{figure}[htbp] 
\scalebox{1.5}{%
\begin{overpic}[scale=0.5, width=0.65\textwidth, tics=2]{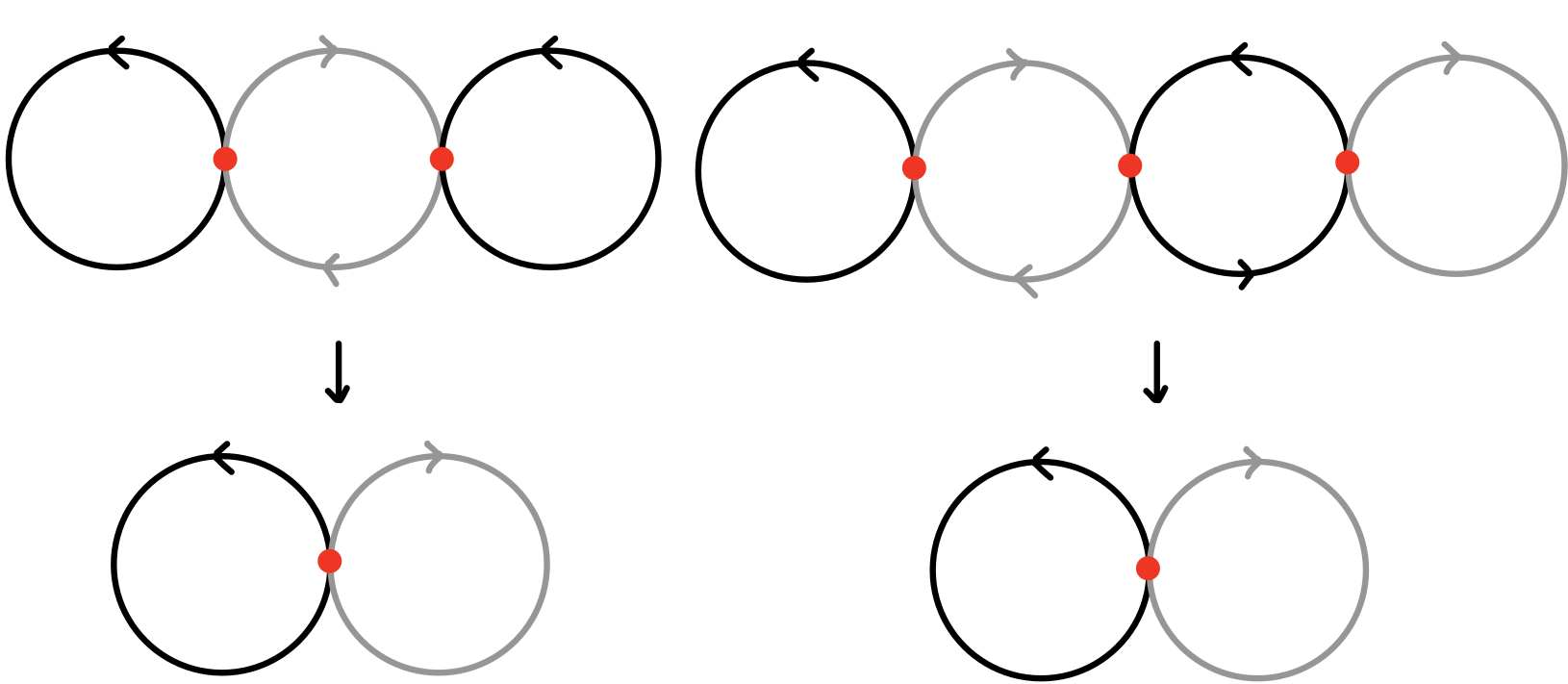}

	\put (21,7)  { \footnotesize\textcolor{red}{$x_0$}}	
	\put (11,11)  { \footnotesize$a$}
	\put (27,11)  { \footnotesize\textcolor{gray}{$b$}}

	\put (22,19)  { \footnotesize$\Pi$}	
	\put (1,7)  { \footnotesize$X$}

	\put (14,33)  { \footnotesize\textcolor{red}{$\widetilde{x}_0$}}	
	\put (5,37)  { \footnotesize$a_1$}
	\put (20,37)  { \footnotesize\textcolor{gray}{$b_1$}}
	\put (20,28)  { \footnotesize\textcolor{gray}{$b_2$}}	
	\put (32,37)  { \footnotesize$a_2$}
	\put (-5,33)  { \footnotesize$\widetilde{X}$}

	\put (73,7)  { \footnotesize\textcolor{red}{$x_0$}}	
	\put (64,10)  { \footnotesize$a$}
	\put (79,10)  { \footnotesize\textcolor{gray}{$b$}}
	
	\put (74,19)  { \footnotesize$\Pi$}	
	\put (87,7)  { \footnotesize$X$}

	\put (58,32)  { \footnotesize\textcolor{red}{$\widetilde{x}_0$}}	
	\put (49,36)  { \footnotesize$a_1$}
	\put (64,37)  { \footnotesize\textcolor{gray}{$b_1$}}
	\put (64,27)  { \footnotesize\textcolor{gray}{$b_2$}}	
		\put (76,37)  { \footnotesize$a_2$}
		\put (76,28)  { \footnotesize$a_3$}
		\put (92,37)  { \footnotesize\textcolor{gray}{$b_3$}}
	\put (99,33)  { \footnotesize$\widetilde{X}$}

   \end{overpic}}	
   \caption{Proof of Lemma \ref{lem: free generators of subgroups}. On the left, the case $k=3$. The loop $a$ has two lifts $a_1$ and $a_2$ which are loops, while $b$ lifts to two non-closed paths $b_1$ and $b_2$. The fundamental group $\pi_1(\widetilde{X},\widetilde{x}_0)$ is freely generated by the homotopy classes of the concatenated paths $a_1$, $b_1\star b_2$, and $b_1\star a_2\star \overline{b}_1$ (where $\overline{\beta}$ denotes the path $\beta$ travelled with the opposite orientation). In this case, one may take $p=1$. On the right, the case $k=4$. Each generator in $\{a,b\}$ has two open lifts, and a closed one. A free generator of $\pi_1(\widetilde{X},\widetilde{x}_0)$ is induced by $a_1$, $b_1\star b_2$, $b_1\star a_3\star a_2 \star \overline{b}_1$, and $b_1\star a_3\star b_3 \star \overline{a}_3\star \overline{b}_1 $, so $p=2$ in this case.}
   \label{fig: cover}

\end{figure}
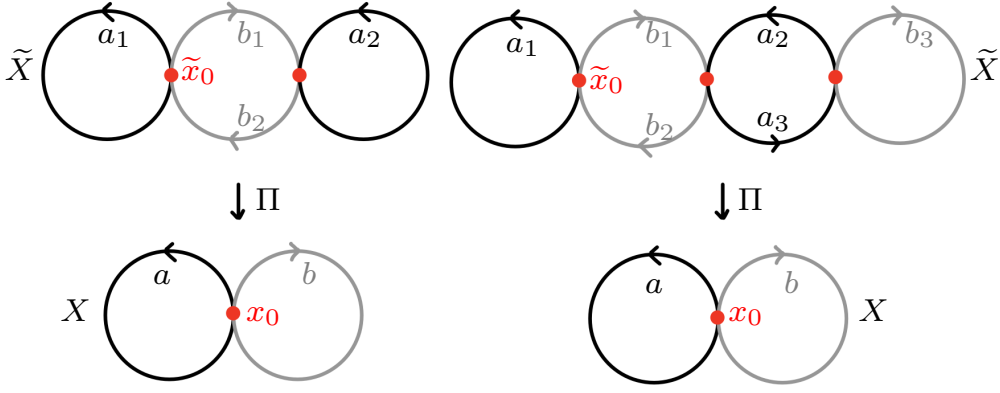

We are now ready to construct our examples. 
Consider the free group $\langle f,g\rangle
$ where $f$ and $g$ are as in Subsection \ref{ss.minimal}, and let $\rho:\langle f,g\rangle\to\SL_3(\R)$ be the inclusion representation.
Fix $k>2$ and $\Gamma_k$ and $F_k$ as in Lemma \ref{lem: free generators of subgroups}.
The identification $\Gamma\cong\langle f,g\rangle$ given by $$a\mapsto g, \text{ }b\mapsto f$$
\noindent induces a quasi-isometric representation $\Gamma\to\SL_3(\R)$ whose restriction to $\Gamma_k$ we denote by $\rho_k$.
By Remark \ref{rem: finite index}, $\rho_k$ is quasi-isometric.
Moreover, as $g\in \rho_k(\Gamma_k)$, the argument of Subsection \ref{ss.minimal} still applies to show that $\rho_k$ is not a limit of Anosov representations. 
Observe moreover that $\rho_k$ can be chosen to be Zariski dense.

To finish, we show that arbitrarily small perturbations of $\rho_k$ contain unipotent elements, and therefore $\rho_k$ is not robustly quasi-isometric (recall Remark \ref{rem: unipotents}).
These perturbations are constructed in two steps. 

First, we consider an arbitrary small perturbation $\rho_k'$ of $\rho_k$ similar to what we did in Subsection \ref{ss.minimal}.
Namely, we let $\rho'_k(c_1)$ to act minimally on the projective line associated to $P_0$ and take $\rho'_k(c_i):=\rho_k(c_i)$ for all $i=2,\dots,k$.
In particular, the induced $\rho'_k$-action of $\langle c_1,c_2\rangle\subset\Gamma_k$ on the Grassmannian $\Gr_2(\R^3)$ of two-dimensional subspaces of $\R^3$ is minimal.

On the other hand, observe that there is some positive integer $q$ such that $\rho'_k(c_3^q)$ preserves a hyperplane $\widehat{P}$ on which acts as a scalar multiple of the identity. 
We may consider a sequence $\{\gamma_n\}\subset \langle c_1,c_2\rangle$ such that $$\rho_k'(\gamma_n)\cdot P_0\to\widehat{P}$$ \noindent as $n\to\infty$.
In particular, we may find an arbitrary small perturbation $\rho_k''$ of $\rho_k'$ with $\rho''_k(c_i)=\rho'_k(c_i)$ for all $i\neq 3$ and such that, for an appropriate large enough $n_0$, the element $\rho''_k(c_3^q)$ preserves 
$$\rho_k''(\gamma_{n_0})\cdot P_0=\rho_k'(\gamma_{n_0})\cdot P_0$$ \noindent and acts as a scalar multiple of the identity on it. 
The commutator $$[c_3^q,\gamma_{n_0}c_1\gamma_{n_0}^{-1}]$$
\noindent has then unipotent image under $\rho_k''$.

\section{Derived from Barbot representations}\label{sec. derived from barbot}

In this section we focus on reducible suspensions $\rho:\Gamma\to\SL_3(\R)$, where $\Gamma$ is a non abelian free group. 
We will use throughout the notations from Equation (\ref{eq. reducible suspension}).
Our main goal is to prove Theorem \ref{t.main}, for which we first need better understanding of which reducible suspensions are quasi-isometric (Proposition \ref{p.equiv}).

\subsection{Proof of Proposition \ref{p.equiv}}\label{ss.proofpropequiv}

The proof of Proposition \ref{p.equiv} is split into Lemma \ref{lem:DFBisQI} and Corollaries \ref{cor:redsuspensionQIisDFB} and \ref{cor:DFBandrobQI} below.

\begin{lemma}\label{lem:DFBisQI}
Let $\Gamma$ be a finitely generated non-abelian free group and \(\rho:\Gamma \to \SL_3(\R)\) be a reducible suspension.
If \(\rho\) is derived from Barbot, then it is quasi-isometric.
\end{lemma}

This is a special case of Lahn \cite[Proposition 3.4]{lahn}, we include a proof for completeness.  

\begin{proof}[Proof of Lemma \ref{lem:DFBisQI}]


We use the notation of Equation (\ref{eq. reducible suspension}).
Let $\Vert\cdot\Vert$ denote the standard Euclidean norm on $\RR^3$.
By abuse of notations, we use the same symbol for the induced norm on $P$ and for the corresponding operator norms. 
Then $$\Vert \rho(\gamma)\Vert\geq \max\{ e^{-\frac{1}{2}\varphi(\gamma)}\Vert\rho_P(\gamma)\Vert, e^{\varphi(\gamma)} \}$$ \noindent for all $\gamma\in\Gamma$. As $\rho_P$ is quasi-isometric, by Equation (\ref{eq: QI in SLd}) we have $$\Vert \rho(\gamma)\Vert\geq \max\{ Be^{-\frac{1}{2}\varphi(\gamma)+A\vert\gamma\vert}, e^{\varphi(\gamma)} \}$$ \noindent for all $\gamma\in\Gamma$ and some $0<A,B\leq 1$.
Up to changing $A$ by a smaller constant if necessary, we have $\Vert \rho(\gamma)\Vert\geq Be^{A\vert\gamma\vert}$ for all $\gamma\in\Gamma$.
\end{proof}

We now turn to the converse statement (Corollary \ref{cor:redsuspensionQIisDFB} below).

\begin{lemma}\label{lem: restriction to plane purely hyperbolic}

Let $\Gamma$ be a finitely generated non-abelian free group and \(\rho:\Gamma \to \SL_3(\R)\) be a reducible suspension preserving some hyperplane $P$.
If \(\rho\) is quasi-isometric, then for every $1\neq \gamma\in\Gamma$ the element $\rho_P(\gamma)$ is hyperbolic.
\end{lemma}

\begin{proof}

We keep the notations from Equation (\ref{eq. reducible suspension}) and suppose by contradiction that $-2\leq \text{tr}(\rho_P(\gamma))\leq 2$ for some $\gamma\neq 1$.
We fix any element $\eta\in\Gamma$ not belonging to the largest cyclic subgroup of $\Gamma$ containing $\gamma$. 
We arrive to the desired contradiction by a case-by-case analysis.
 
On the one hand, we first note that $\rho_P(\gamma)\neq\pm\text{id} $, otherwise the matrix $\rho([\gamma,\eta])$ would be unipotent contradicting Remark \ref{rem: unipotents}.
Hence, it only remains to rule out the possibility of $\rho_P(\gamma)$ being either parabolic or elliptic. 
We only treat the first case, the second one can be handled analogously. 

Suppose then that $\text{tr}(\rho_P(\gamma))=\pm 2$ and $\rho_P(\gamma)\neq \pm \text{id}$.
Up to replacing $\gamma$ by $\gamma^2$ if needed, we may assume $\text{tr}(\rho_P(\gamma))= 2$.
Hence, we may replace $\rho$ by a conjugate representation in order to have $$\rho(\gamma)=\left( \begin{matrix}
e^{-\frac{1}{2}\varphi(\gamma)} \left(\begin{smallmatrix}
1 & \pm 1\\
0 &  1 \end{smallmatrix}\right) & \kappa(\gamma)\\
0 &  e^{\varphi(\gamma)}
\end{matrix} \right).$$
\noindent Note then that $\varphi(\gamma)$ must be different from $0$, otherwise we get a contradiction with Remark \ref{rem: unipotents}.
We write $$\rho(\gamma)=\left( \begin{matrix}
\tau \cdot g & 0\\
0 &  \tau^{-2}
\end{matrix} \right)$$ \noindent for simplicity, where $\tau:=e^{-\frac{1}{2}\varphi(\gamma)}$ and $g:= \left(\begin{smallmatrix}
1 & \pm 1\\
0 &  1 \end{smallmatrix}\right)$. 
Up to replacing $\gamma$ by $\gamma^{-1}$, we may assume $\tau <1$. 

We now claim that the norm $\Vert\rho(\gamma^n\eta\gamma^{-n})\Vert$ does not grow exponentially as $n\to\infty$, which will give the desired contradiction.
Indeed, we have $$\rho(\eta)=\left( \begin{matrix}
\delta \cdot h & \kappa\\
0 &  \pm\delta^{-2}
\end{matrix} \right),$$ \noindent for some $h\in\SL_2^\pm(\R)$, a real number $\delta > 0$, and a vector $\kappa\in\R^2$.
Hence, $$\rho(\gamma^n\eta\gamma^{-n})=\left( \begin{matrix}
\delta \cdot g^nhg^{-n} & \tau^{3n} g^{n}\cdot\kappa\\
0 &  \pm\delta^{-2}
\end{matrix} \right).$$ \noindent As the norms $\Vert g^nhg^{-n} \Vert $ and $\Vert  g^{n} \Vert $ grow at most polynomially as $n\to\infty$ and $\tau <1$, the proof is complete.

\end{proof}

\begin{cor}\label{cor:redsuspensionQIisDFB}
Let $\Gamma$ be a finitely generated non-abelian free group and \(\rho:\Gamma \to \SL_3(\R)\) be a reducible suspension.
If \(\rho\) is quasi-isometric, then it is derived from Barbot.
\end{cor}

\begin{proof}

We need to show that $\rho_P$ is quasi-isometric, for which we first show that it is faithful and discrete (compare with Lahn \cite[Proposition 3.1]{lahn}).
Indeed, Lemma \ref{lem: restriction to plane purely hyperbolic} directly implies that $\rho_P$ is faithful.
With this at hand we can show that it is discrete.
Otherwise, $\rho_P(\Gamma)$ would contain an elliptic element (c.f. the proof of Proposition \ref{prop:robustfaithfulSL2}), and this would contradict again Lemma \ref{lem: restriction to plane purely hyperbolic}. 

Finally, by Remark \ref{rem: finite index} in order to show that $\rho_P$ is quasi-isometric it suffices to show that its restriction to $$\Gamma_0:=\{\gamma\in\Gamma: \rho_P(\gamma)\in\SL_2(\R)\}$$ \noindent is quasi-isometric.
But note that the representation $\rho_P\vert_{\Gamma_0}:\Gamma_0\to\SL_2(\R)$ is faithful and discrete, and by Lemma \ref{lem: restriction to plane purely hyperbolic} does not contain parabolic elements. 
As in the proof of Proposition \ref{prop:robustfaithfulSL2}, this implies that $\rho_P\vert_{\Gamma_0}$ is convex co-compact.

\end{proof}

To finish the proof of Proposition \ref{p.equiv} we show the following.

\begin{cor}\label{cor:DFBandrobQI}
Let $\Gamma$ be a finitely generated non-abelian free group and \(\rho:\Gamma \to \SL_3(\R)\) be a reducible suspension preserving some hyperplane $P$.
Then \(\rho\) is derived from Barbot if and only if \(\rho\) is robustly quasi-isometric among representations preserving \(P\).
\end{cor}

\begin{proof}

This follows from Lemma \ref{lem:DFBisQI} and Corollary \ref{cor:redsuspensionQIisDFB}, and the fact that quasi-isometric representations into $\SL^\pm(P)$ form a stable class.
Indeed, every small perturbation of $\rho$ preserving $P$ induces a small perturbation of $\rho_P$ and conversely.
\end{proof}

\subsection{Finding eigenvalues with repeated moduli}\label{ss.multiplicity}

The goal of this subsection is to prove Proposition \ref{prop-invariantplanejordan} below, which together with Proposition \ref{prop: criterion for non robust qi} is the main step in the proof of Theorem \ref{t.main}.

For $g \in \SL_3(\RR)$ preserving the hyperplane $P$ we let $$\lambda_{1}(g)\geq \lambda_{2}(g)$$ \noindent be the moduli of the eigenvalues of $g$ restricted to $P$. 
We also let $$ \lambda_{\perp}(g) := \frac{1}{\lambda_{1}(g)\cdot\lambda_{2}(g)} $$ \noindent be the modulus of the complementary eigenvalue of $g$. 
Observe that $\lambda_\perp$ is a group morphism: if $h\in\SL_3(\R)$ also preserves $P$, then $$\lambda_{\perp}(gh)=\lambda_{\perp}(g)\cdot \lambda_{\perp}(h).$$ 
Note also that if $\lambda_1(g)>\lambda_2(g)$ then all the eigenvalues of $g$ are real and the corresponding Jordan blocks have size at most two.

\begin{prop}\label{prop-invariantplanejordan}
Suppose that $\Gamma$ is a non-abelian free group of rank $k\geq 2$, and let $\rho: \Gamma \to \SL_3(\RR)$ be a derived from Barbot representation preserving the hyperplane $P$.
Assume moreover that $\rho_P$ preserves the orientation of $P$.
Suppose that $\rho$ is not Anosov and fix some neighborhood $\mathscr{U}$ of $\rho$.
Then there exists $\rho'\in\mathscr{U}$ preserving $P$, and free generators $a=c_1,b=c_2,c_3,\dots, c_k$ of $\Gamma$ such that the following holds:

\begin{enumerate}
\item there exist non-zero integers $n$ and $m$ so that $\rho'(a^mb^n)$ satisfies $$ \lambda_1(\rho'(a^mb^n))=\lambda_\perp(\rho'(a^mb^n))\neq 1. $$
\item The matrices $\rho'(a)$ and $\rho'(b)$ are loxodromic, and the eigenline $L_0$ corresponding to $\lambda_2(\rho'(a^mb^n))$ is contained in the repelling hyperplane $\Ecs(\rho'(a))$.
\end{enumerate}
\end{prop}

The proof of Proposition \ref{prop-invariantplanejordan} involves several intermediate results.
Roughly, the outline is the following. 
We first prove in Proposition \ref{prop. generator in V} that Lahn's Theorem \ref{teo: lahn in examples} implies that, after possibly perturbing $\rho$ slightly inside the set of representations preserving $P$, there is a free generating set of $\Gamma$ containing an element $a$ so that $$\lambda_2(\rho(a))<\lambda_1(\rho(a))<\lambda_\perp(\rho(a)).$$
\noindent We then show in Proposition \ref{prop. generator in V and complementary not in V} that we may also assume that a complementary generating element $b$ satisfies $$\lambda_2(\rho(b))<\lambda_\perp(\rho(b))<\lambda_1(\rho(b)).$$
\noindent The fact that $\rho$ is derived from Barbot together with a continuity argument will then prove the statement.

We assume throughout that $\Gamma$ is a non abelian free group of rank $k\geq 2$, and fix a free generating ordered set $F_0\subset\Gamma$, which will serve as a reference to construct the one in Proposition \ref{prop-invariantplanejordan}.
This choice induces a morphism $$\psi_0:\Gamma\to \ZZ^k,$$ \noindent by identifying the abelianization $H^1(\Gamma):=\Gamma/[\Gamma,\Gamma]$ of $\Gamma$ with $\ZZ^k$.
We will say that two elements in $\Gamma$ are \textit{homologous} if they have the same image under $\psi_0$.

We have the following properties of derived from Barbot representations that we will use all the time (recall the notations from Equation (\ref{eq. reducible suspension})).

\begin{rem}\label{rem: scaling generator in P}
Suppose that $\rho:\Gamma\to\SL_3(\R)$ is a derived from Barbot representation preserving $P$. 
Then:
\begin{enumerate}
\item for all $\gamma\in\Gamma$ we have $$\lambda_{1}(\rho(\gamma))> \lambda_{2}(\rho(\gamma)).$$
\item The largest modulus among the eigenvalues of $\rho_P(\gamma)$ is $\lambda_\perp^{\frac{1}{2}}(\rho(\gamma))\cdot\lambda_1(\rho(\gamma))$.
In particular, as $\rho_P$ is quasi-isometric we have $$\displaystyle\inf_{\gamma\in\Gamma}\lambda_\perp^{\frac{1}{2}}(\rho(\gamma))\cdot\lambda_1(\rho(\gamma))>1.$$
\item Let $\gamma_1,\dots,\gamma_s\in\Gamma$ be elements so that the eigenlines (in $P$) of the $\rho_P(\gamma_i)$ are all different. 
Then $$\frac{\lambda_1(\rho(\gamma_1^{m_1}\dots\gamma_s^{m_s}))}{\lambda_1(\rho(\gamma_1))^{m_1}\dots\lambda_1(\rho(\gamma_s))^{m_s}}$$
\noindent converges to some positive constant only depending on the eigenlines of the $\rho_P(\gamma_i)$, as $m_1,\dots,m_s\to+\infty$ (see e.g. \cite[Lemma 7.5]{benoistquint}).
\item Fix some $\eps$ and choose a generator $c_0\in F_0$.
We may consider another derived from Barbot representation, which is obtained by \emph{scaling the $c_0$-action on $P$ by the factor $\eps$}. 
More precisely, let $\rho_\eps:\Gamma\to\SL_3(\R)$ be the representation given by $$\rho_\eps(c_0):=\left( \begin{matrix}
e^{-\frac{1}{2}(\varphi(c_0)+\eps)}\rho_P(c_0) & \kappa(c_0)\\
0 & \pm e^{\varphi(c_0)+\eps}
\end{matrix} \right),$$ \noindent and $\rho_\eps(c):=\rho(c)$ for all $c\in F_0\setminus\{c_0\}$.

For $\gamma\in\Gamma$ there is an integer $p=p(c_0,\gamma)$ which is the coordinate of $\psi_0(\gamma)$ associated to the generator $c_0$.
One has $$\lambda_\perp(\rho_\eps(\gamma))=e^{\eps p}\cdot\lambda_\perp(\rho(\gamma)) \text{ and } \lambda_1(\rho_\eps(\gamma))=e^{-\frac{1}{2}\eps p}\cdot\lambda_1(\rho(\gamma)).$$
\noindent In particular, $$\frac{\lambda_\perp(\rho_\eps(\gamma))}{\lambda_1(\rho_\eps(\gamma))}=e^{\frac{3}{2}\eps p}\cdot \frac{\lambda_\perp(\rho(\gamma))}{\lambda_1(\rho(\gamma))}.$$
\end{enumerate} 
\end{rem}

Here is a consequence of Lahn's Theorem \ref{teo: lahn in examples} that we will need in the future.

\begin{cor}\label{cor. lahn with perturbation}
Let $\rho: \Gamma \to \SL_3(\RR)$ be a derived from Barbot representation preserving the hyperplane $P$.
Suppose that $\rho$ is not Anosov and fix some neighborhood $\mathscr{U}$ of $\rho$.
Then there exists $\rho'\in\mathscr{U}$ preserving $P$ and $\gamma\in\Gamma$ so that $$\lambda_1(\rho'(\gamma))<\lambda_\perp(\rho'(\gamma)).$$ 
\end{cor}

\begin{proof}

By Theorem \ref{teo: lahn in examples} we may find a sequence $\gamma_n\in\Gamma$ with $\varphi(\gamma_n)>0$ and some $d\leq 1$ so that $$\displaystyle\lim_{n\to+\infty}\frac{\log\lambda_1(\rho(\gamma_n))}{\log\lambda_\perp(\rho(\gamma_n))}=d.$$
\noindent If $d<1$ there is nothing to prove, so we assume \begin{equation}\label{eq. lahn perturbed}
\displaystyle\lim_{n\to+\infty}\frac{\log\lambda_1(\rho(\gamma_n))}{\log\lambda_\perp(\rho(\gamma_n))}=1.
\end{equation}
\noindent Observe that $\psi_0(\gamma_n)\neq 0$ for all $n$, as $\log\lambda_\perp(\rho(\gamma_n))>0$.
The proof will be split into several cases. 

First assume that the sequence $\{\log\lambda_\perp(\rho(\gamma_n))\}_n$ is bounded away from $0$ and $\infty$. 
Up to replacing $\gamma_n$ by appropriate conjugates if needed, this implies that the sequence $\{\gamma_n\}_n$ is bounded, and therefore it has some subsequence equal to some $\gamma\in\Gamma$.
By Equation (\ref{eq. lahn perturbed}) we have then $$\frac{\log\lambda_1(\rho(\gamma))}{\log\lambda_\perp(\rho(\gamma))}=1.$$
\noindent As $\psi_0(\gamma)\neq 0$, we may scale the action of some element in $F_0$ to obtain the desired result (c.f. Remark \ref{rem: scaling generator in P}).

Secondly, observe that by Equation (\ref{eq. lahn perturbed}) and Remark \ref{rem: scaling generator in P} no subsequence of $\{\log\lambda_\perp(\rho(\gamma_n))\}_n$ converges to $0$. 
Hence, to finish the proof it remains to treat the case $\log\lambda_\perp(\rho(\gamma_n))\to\infty$ as $n\to\infty$.

Write $\psi_0(\gamma_n)=(p_{1,n},\dots,p_{k,n})$ for all $n$.
Up to reordering the generator $F_0$ we may assume that $\vert p_{1,n}\vert\to\infty$ as $n\to\infty$ and \begin{equation}\label{eq. lahn perturbed II}
\displaystyle\lim_{n\to\infty}\frac{\log\lambda_\perp(\rho(\gamma_n))}{p_{1,n}}=\alpha,
\end{equation} \noindent for some $\alpha\in\R$.

Up to taking a further subsequence if necessary, we may assume that there is some small $\varepsilon$ so that $\varepsilon \cdot p_{1,n}>0$ for all $n$.  
We now scale the action of the first generator of $F_0$ as in Remark \ref{rem: scaling generator in P}. 
We obtain $$\displaystyle\limsup_{n\to\infty}\frac{\log\lambda_1(\rho_\varepsilon(\gamma_n))}{\log\lambda_\perp(\rho_\varepsilon(\gamma_n))}\leq \displaystyle\limsup_{n\to\infty}\frac{\log\lambda_1(\rho(\gamma_n))}{\varepsilon \cdot p_{1,n}+\log\lambda_\perp(\rho(\gamma_n))},$$ \noindent which by Equations (\ref{eq. lahn perturbed}) and (\ref{eq. lahn perturbed II}) is equal to $$\frac{1}{\varepsilon/\alpha+1}<1.$$
\noindent This finishes the proof.

\end{proof}

The result above suggest us to introduce the set $$V_\rho:=\{\gamma\in\Gamma: \lambda_1(\rho(\gamma)) < \lambda_\perp(\rho(\gamma))\},$$ \noindent where $\rho$ is derived from Barbot.
Note that $$V_\rho^{-1}=\{\gamma\in\Gamma: \lambda_\perp(\rho(\gamma))< \lambda_2(\rho(\gamma))\}.$$
\noindent Observe that $V_\rho$ is invariant under conjugacy, and under taking positive powers. 
We also have the following, which says that decreasing the $\rho_P$-length of a curve keeping the homology class fixed keeps you inside $V_\rho$. 

\begin{lem}\label{lem: short representatives in V}
Let $\rho: \Gamma \to \SL_3(\RR)$ be a derived from Barbot representation preserving the hyperplane $P$.
Consider homologous elements $\gamma_0,\gamma_1\in\Gamma$ so that $\gamma_0\in V_\rho$ and $$\frac{\lambda_1(\rho(\gamma_1))}{\lambda_2(\rho(\gamma_1))}\leq \frac{\lambda_1(\rho(\gamma_0))}{\lambda_2(\rho(\gamma_0))}.$$ \noindent Then $\gamma_1$ belongs to $V_\rho$.  
\end{lem}

\begin{proof}

Indeed, on the one hand as $\psi_0(\gamma_1)=\psi_0(\gamma_0)$ we have $\lambda_\perp(\rho(\gamma_1))=
\lambda_\perp(\rho(\gamma_0))$.
Also, this fact together with the inequality in the statement implies $$\lambda_2(\rho(\gamma_0))\leq \lambda_2(\rho(\gamma_1))< \lambda_1(\rho(\gamma_1))\leq \lambda_1(\rho(\gamma_0)).$$ \noindent This proves the lemma.
\end{proof}

From the above lemma we manage to find generating elements in $V_{\rho'}$, for a small perturbation $\rho'$ of $\rho$ (this is the only place in our argument were we need to assume that the restriction $\rho_P$ preserves the orientation of $P$).

\begin{prop}\label{prop. generator in V}
Let $\rho: \Gamma \to \SL_3(\RR)$ be a derived from Barbot representation preserving the hyperplane $P$.
Assume moreover that $\rho_P$ preserves the orientation of $P$.
Suppose that $\rho$ is not Anosov and fix some neighborhood $\mathscr{U}$ of $\rho$.
Then there exists some $\rho'\in\mathscr{U}$ preserving $P$ and a free generating set $F=\{a=c_1,c_2,\dots,c_k\}\subset \Gamma$ such that $a\in V_{\rho'}$.
\end{prop}

To prove Proposition \ref{prop. generator in V} we look at the following useful object. 
Consider the convex co-compact (orientable) hyperbolic surface $S:=\HH^2/\rho_P(\Gamma)$.
We let $C\subset S$ be the convex core, which is a compact hyperbolic surface with geodesic boundary containing all closed geodesics of $S$.

For a given non-zero $h\in H^1(\Gamma)=\Gamma/[\Gamma,\Gamma]$ we consider the infimum $\Vert h\Vert$ of lengths of multi-curves in $C$ representing the homology class $h$.
The Arzel\'a-Ascoli Theorem implies that this is actually a minimum.
Let $m_h$ be a minimizing multi-curve. 
Its components are not all reduced to points, otherwise we would have $h= 0$. 
In particular, $\Vert h \Vert >0$ and we may assume that all the components of $m_h$ are (non constant) closed curves.
Each of these curves are geodesics, as they minimize the length in its homotopy class.
Further, these geodesics are simple (considered possibly with multiplicity) and their union is disjoint, c.f. McShane-Rivin \cite[Theorem 4.1]{mcshane-rivin}.
None of these simple geodesics is homologically trivial, otherwise it could be removed from $m_h$ reducing the length.

\begin{proof}[Proof of Proposition \ref{prop. generator in V}]
By Corollary \ref{cor. lahn with perturbation} we may assume that there is some $\gamma_0\in V_\rho$. 
As a first step we prove that there is an element in $V_\rho$ which induces a simple closed geodesic in $C$.
Afterwards we show that this suffices to prove the statement.

Let us then prove the first claim, that is, we will find a simple geodesic in $V_\rho$.
For this purpose, we consider the homology class $h:=\psi_0(\gamma_0)$, which is non-zero as $\gamma_0\in V_\rho$.
In particular, if our element $\gamma_0\in V_\rho$ is itself is a minimizing multi-curve, then as it is a curve, it is simple and there is nothing to prove.
As a consequence, to prove our first claim it remains to treat the case in which $$\Vert\psi_0(\gamma_0)\Vert<\ell(\gamma_0),$$ \noindent where for $\gamma\in\Gamma$ we let $$\ell(\gamma):=\log\left( \frac{\lambda_1(\rho(\gamma))}{\lambda_2(\rho(\gamma))}\right)$$ \noindent denote the length of the closed geodesic in $C$ associated to $\gamma\in\Gamma$.
We then have $$\displaystyle\lim_{k\to+\infty}\ell(\gamma_0^k)-k\cdot \Vert \psi_0(\gamma_0) \Vert=\displaystyle\lim_{k\to+\infty}k\cdot(\ell(\gamma_0)-\Vert\psi_0(\gamma_0)\Vert)=+\infty.$$
\noindent On the other hand, let $\gamma_1,\dots,\gamma_s$ be the (disjoint) simple closed geodesics supporting a length minimizing multi-curve in $\psi_0(\gamma_0)$ (considered possibly with multiplicity).
For each positive integer $k$ consider the curve $\widehat{\gamma}_k:=\gamma_1^k\dots\gamma_s^k$, which is homologous to $\gamma_0^k$.
By Remark \ref{rem: scaling generator in P} we have $$\displaystyle\limsup_{k\to+\infty} \left(\ell(\widehat{\gamma}_k)-k\cdot \Vert\psi_0(\gamma_0)\Vert\right)<+\infty.$$ 
\noindent By Lemma \ref{lem: short representatives in V} we conclude that $\widehat{\gamma}_k\in V_\rho$ for every $k$ large enough.
That is, $$1>\frac{\lambda_1(\rho(\widehat{\gamma}_k))}{\lambda_\perp(\rho(\widehat{\gamma}_k))}=\frac{\lambda_1(\rho(\widehat{\gamma}_k))}{\lambda_1(\rho(\gamma_1))^k\dots\lambda_1(\rho(\gamma_s))^k}\cdot 
\left(\frac{\lambda_1(\rho(\gamma_1))}{\lambda_\perp(\rho(\gamma_1))}\right)^k\dots \left(\frac{\lambda_1(\rho(\gamma_s))}{\lambda_\perp(\rho(\gamma_s))}\right)^k$$ \noindent for all $k$ large enough. 
Applying again Remark \ref{rem: scaling generator in P} we conclude that $\gamma_i\in V_\rho$ for some $i=1,\dots,s$, thus proving our first claim.

From the discussion above, we may assume that $\gamma_0\in V_\rho$ induces in fact a simple (non homologically trivial) closed geodesic in $C\subset S$. 
As we now show this implies that $\gamma_0$ belongs to a free generating set of $\Gamma$, thus finishing the proof.

Indeed, this claim is a general fact about (orientable) topological surfaces of negative Euler characteristic. 
Let $g\geq 0$ be the genus of $C$ and $n\geq 1$ be the number of boundary components. 
If $\gamma_0$ induces a boundary curve of $C$, then as it is non homologically trivial, we necessarily have $n\geq 2$. 
In this case $\gamma_0$ belongs to a free generating set of $\Gamma$ and there is nothing to prove.
More generally, if $\gamma_0$ is non separating, by classification of surfaces it also belongs to a free generating set. 
Hence we assume from now on that $\gamma_0$ separates $C$ in two compact hyperbolic surfaces $C_1$ and $C_2$.
In particular, $\Gamma$ splits as the amalgamated product $$\Gamma_1\star_{\langle\gamma_0\rangle}\Gamma_2,$$ \noindent where $\Gamma_i$ denotes the fundamental group of $C_i$.

We observe that both $C_1$ and $C_2$ have more than one boundary components, as $\gamma_0$ is not homologically trivial in $C$. 
Hence, for $i=1,2$ the group $\Gamma_i$ admits a free generating set $F_i:=\{a_1,b_1,\dots,a_{g_i},b_{g_i},c_1\dots,c_{n_i}\}$ for some $g_i\geq 0$ and $n_i\geq 1$ so that $$\gamma_0=\left(\displaystyle\prod_{j=1}^{g_i} [a_j,b_j]\right)c_1\dots c_{n_i}.	$$
\noindent In particular, $$F_i':=\{a_1,b_1,\dots,a_{g_i},b_{g_i},c_1\dots,
c_{n_i-1},\gamma_0\}$$ \noindent is a free generating set of $\Gamma_i$.
Hence, $F:=F_1'\cup F_2'$ is a free generating set of $\Gamma$ and contains $\gamma_0\in V_{\rho}$.

\end{proof}

We now show that a complementary generating element can be assumed to be outside $V_{\rho'}\cup V_{\rho'}^{-1}$.

\begin{prop}\label{prop. generator in V and complementary not in V}
Let $\rho: \Gamma \to \SL_3(\RR)$ be a derived from Barbot representation preserving the hyperplane $P$. 
Assume moreover that $\rho_P$ preserves the orientation of $P$.
Suppose that $\rho$ is not Anosov and fix some neighborhood $\mathscr{U}$ of $\rho$.
Then there exists some $\rho'\in\mathscr{U}$  preserving $P$ and a free generating set $\{a=c_1,b=c_2,c_3,\dots,c_k\}\subset \Gamma$ such that $$a\in V_{\rho'} \text{ and } b\notin V_{\rho'}\cup V_{\rho'}^{-1}.$$
\end{prop}

\begin{proof}
By Proposition \ref{prop. generator in V}, up to replacing $\rho$ by a small perturbation if needed we may find a free generating set $F=\{a=c_1,c_2,\dots,c_k\}$ of $\Gamma$ so that $a\in V_\rho$.
If for some $i=2,\dots,k$ we have that $c_i$ does not belong to $ V_{\rho}\cup V_{\rho}^{-1}$ there is nothing to prove, so we assume that this is not the case.
Let $b:=c_2$.
We will construct an automorphism  $\sigma$ of $\Gamma$ such that the generating set $\sigma(F)$ satisfies $$\sigma(a)\in V_{\rho} \text{ and } \sigma(b)\notin V_{\rho}\cup V_{\rho}^{-1}.$$

The construction is inductive. 
For the first step, note that up to replacing $b$ by $b^{-1}$ we may assume $b\in V_\rho$. 
In particular, both $\lambda_\perp(\rho(a))$ and $\lambda_\perp(\rho(b))$ are strictly larger than $1$. 
Up to composing with the automorphism of $\Gamma$ that permutes $a$ with $b$ and fixes all the other generators if needed, we may assume that $$\tau(a)\leq\tau(b)<1,$$ \noindent where $\tau:\Gamma\to\R_{>0}$ is the morphism given by $$\tau(\gamma):=
\lambda_\perp(\rho(\gamma))^{-1}.$$
\noindent This gives us an automorphism $\sigma_1$ such that both $\sigma_1(a)$ and $\sigma_1(b)$ belong to $V_\rho$, and moreover $$\tau(\sigma_1(a))\leq\tau(\sigma_1(b))<1.$$

Let now $n$ be an integer $\geq 1$.
We assume by induction that we have constructed an automorphism $\sigma_n$ of $\Gamma$ such that both $\sigma_n(a)$ and $\sigma_n(b)$ belong to $V_\rho$, $\tau(\sigma_n(a))\leq\tau(\sigma_n(b))<1$, and moreover $$\sigma_n(a)=\sigma_{n-1}(b).$$
\noindent We find a positive integer $p_n$ so that $$\tau(\sigma_n(b))<\tau(\sigma_n(ab^{-p_n}))\leq 1,$$ \noindent and we define the automorphism $\sigma_{n+1}$ of $\Gamma$ by letting $$\sigma_{n+1}(a):=\sigma_n(b) \text{ and } \sigma_{n+1}(b):=\sigma_n(ab^{-p_n}),$$ 
\noindent and keeping fixed all the other generators.
Note then that $\sigma_{n+1}(a)\in V_\rho$. 
Hence, if $\sigma
_{n+1}(b)\notin V_{\rho}\cup V_{\rho}^{-1}$ we are done. 
If this is not the case, we continue the induction. 
We only have to show that the process stops at some finite step.

Suppose by contradiction that this is not the case. 
We get an infinite sequence of automorphisms $\{\sigma_n\}_{n\geq 1}$ such that for all $n$ both $\sigma_n(a)$ and $\sigma_n(b)$ belong to $V_\rho$, $\tau(\sigma_n(a))\leq\tau(\sigma_n(b))<1$, and $\sigma_{n+1}(a)=\sigma_n(b)$.
Furthermore, by construction the sequence $\{\tau(\sigma_n(b))\}$ is strictly increasing and therefore converges.
We claim that in fact converges to $1$.
Indeed, note that $\{\tau(\sigma_n(a))\}_{n\geq 1}$ converges to the same limit, and 
$$\frac{\tau(\sigma_n(a))}{\tau(\sigma_n(b))}=\tau(\sigma_n(ab^{-p_n}))\cdot \tau(\sigma_n(b^{p_n-1}))\leq \tau(\sigma_{n+1}(b)),$$ \noindent for all $n$.
This shows the claim. 

We then have $$\displaystyle\lim_{n\to\infty}\lambda_\perp(\rho(\sigma_n(b)))=1.$$
But by Remark \ref{rem: scaling generator in P} we have $$\displaystyle\limsup_{n\to\infty}\frac{\lambda_1(\rho(\sigma_n(b)))}{\lambda_\perp(\rho(\sigma_n(b)))}=\displaystyle\limsup_{n\to\infty}\frac{\lambda_\perp^{\frac{1}{2}}(\rho(\sigma_n(b)))\cdot \lambda_1(\rho(\sigma_n(b)))}{\lambda_\perp^{\frac{3}{2}}(\rho(\sigma_n(b)))}>1.$$
\noindent We may then find some integer $n>1$ such that $$\frac{\lambda_1(\rho(\sigma_{n}(b)))}{\lambda_\perp(\rho(\sigma_{n}(b)))}>1.$$
\noindent By construction we also have $\lambda_\perp(\rho(\sigma_{n}(b)))> 1$, and this shows that $\sigma_{n}(b)\notin V_{\rho}\cup V_{\rho}^{-1}$, contradicting our assumptions.

\end{proof}

We can finally prove the desired result of this subsection.

\begin{proof}[Proof of Proposition  \ref{prop-invariantplanejordan}] 
By Proposition \ref{prop. generator in V and complementary not in V} there is a free generating set $F$ of $\Gamma$ and $a,b\in F$ so that, after possibly perturbing $\rho$ inside $\mathscr{U}$, we have $$ \lambda_2(\rho(a)) < \lambda_1(\rho(a))< \lambda_\perp (\rho(a)), $$
\noindent and $$  \lambda_2(\rho(b))\leq \lambda_\perp(\rho(b))\leq  \lambda_1 (\rho(b)).$$
\noindent In particular, $P =\Ecs(\rho(a))$. 
Moreover, as $\lambda_2(\rho(b))<  \lambda_1 (\rho(b))$ we can perturb slightly $\rho(b)$ to assume that this matrix is also loxodromic, that is, the inequalities above are strict.

Now, by Remark \ref{rem: scaling generator in P} we have that  $$ \frac{\lambda_1 (\rho(a^{m}b^{n})) }{\lambda_1 (\rho(a))^{m}\lambda_1 (\rho(b))^{n}} $$ \noindent converges to some positive constant only depending on $\rho_P(a)$ and $\rho_P(b)$, as $m,n\to+\infty$ .  
Hence, as $$\frac{ \lambda_1 (\rho(a^{m}b^{n})) }{\lambda_{\perp} (\rho(a^{m} b^{n}))}= \frac{\lambda_1 (\rho(a^{m}b^{n}))}{\lambda_1 (\rho(a))^{m}\cdot \lambda_1 (\rho(b))^{n}}\cdot\left(\frac{\lambda_1 (\rho(a))}{\lambda_\perp (\rho(a))}\right)^m\cdot \left(\frac{\lambda_1 (\rho(b))}{\lambda_\perp (\rho(b))}\right)^n
$$ \noindent and as $ \lambda_\perp (\rho(a)) > \lambda_1(\rho(a)) $ and $\lambda_1(\rho(b)) > \lambda_{\perp}(\rho(b))$, we find a constant $C>1$ so that there are sequences $m_k, n_k \to +\infty$ such that $$ \frac{ \lambda_1 (\rho(a^{m_k}b^{n_k})) }{\lambda_{\perp} (\rho(a^{m_k} b^{n_k}))} \in [C^{-1}, C ]$$ \noindent for all $k$. 

Fix a small $\eps_0>0$ so that for every $\eps\in(-\eps_0,\eps_0)$ the representation $\rho_\eps$ obtained by scaling the $a$-action on $P$ belongs to $\mathscr{U}$ (recall Remark \ref{rem: scaling generator in P}).
We have $$ \frac{ \lambda_1 (\rho_\eps(a^{m_k}b^{n_k})) }{\lambda_{\perp} (\rho_\eps(a^{m_k} b^{n_k}))}=e^{-\frac{3}{2}m_k\eps}\cdot \frac{ \lambda_1 (\rho(a^{m_k}b^{n_k})) }{\lambda_{\perp} (\rho(a^{m_k} b^{n_k}))}\in\left[C^{-1}e^{-\frac{3}{2}m_k\eps},Ce^{-\frac{3}{2}m_k\eps}\right]$$ \noindent for all $k$.
As $m_k$ can be taken to be arbitrarily large, by continuity we find some $\eps$ such that for some $m=m_k$ and $n=n_k$ we have $$\frac{ \lambda_1 (\rho_\eps(a^{m}b^{n})) }{\lambda_{\perp} (\rho_\eps(a^{m} b^{n}))}=1.$$
\noindent Note moreover that $\lambda_1 (\rho_\eps(a^{m}b^{n}))\neq 1$, because $\rho_\eps$ still preserves $P$ and $\rho_P=(\rho_\eps)_P$ is quasi-isometric.
In particular, $L_0\subset P=\Ecs(\rho_\eps(a))$.
Moreover, as the deformation $\rho_\eps$ is arbitrarily small and $\rho(a)$ and $\rho(b)$ were loxodromic, so they are $\rho_\eps(a)$ and $\rho_\eps(b)$.
This finishes the proof.
\end{proof}

\subsection{Proof of Theorem \ref{t.main}} \label{subsec: proof of main thm}

We are now ready to prove Theorem \ref{t.main}, which will be a consequence of Propositions \ref{prop: criterion for non robust qi} and \ref{prop-invariantplanejordan}.

Hence, we fix a derived from Barbot representation $\rho$ of a free group $\Gamma$, so that $\rho_P$ preserves the orientation of $P$.
We suppose moreover that $\rho$ is not Anosov.
We will find arbitrarily small perturbations of $\rho$ which are not quasi-isometric. 

Indeed, by Proposition \ref{prop-invariantplanejordan}, up to replacing $\rho$ by an arbitrarily small deformation still preserving $P$, and letting $\omega:=a^mb^n$ for appropriate integers $m$ and $n$ we have:

\begin{enumerate}
\item the restriction of $\rho(\omega)$ to some hyperplane $P_0$ is either diagonalizable with eigenvalues of equal modulus $\lambda_1=\lambda_\perp\neq 1$, or it is a Jordan block of eigenvalue $\mu$, for some $\mu\neq \pm 1$.
\item The matrices $\rho(a)$ and $\rho(b)$ are loxodromic.
\item The complementary eigenline $L_0$ of $\rho(\omega)$ is contained in $\Ecs(\rho(a))$.
\end{enumerate}

In order to apply Proposition \ref{prop: criterion for non robust qi}, we now show that we may find an arbitrarily small perturbation $\rho'$ of $\rho$, not preserving $P$ anymore, and for which condition (1) is replaced by $$P_0=\ker(\rho' (\omega^q)-\mu^q), $$ \noindent for some non-zero integer $q$ and $\mu\neq\pm 1$, while still keeping conditions (2) and (3) unchanged. 
Indeed, if the restriction of $\rho'(\omega)$ to $P_0$ is diagonalizable then $\rho'=\rho$ and $q=2$ do the job. 
In contrast, if it is a Jordan block instead, by Lemma \ref{lem: power map is open} there exist arbitrarily small neighborhoods $U$ of $\rho(b)$ so that $U^{(n)}$ is a neighborhood of $\rho(b^n)$.
In particular, $\rho(a^m)U^{(n)}$ is a neighborhood of $\rho(\omega)$.
Fix an element $h\in \rho(a^m)U^{(n)}$ preserving the splitting $L_0\oplus P_0$ and for which $$ P_0=\ker(h^q-\mu^q) $$ \noindent for some non-zero integer $q$.
We may take $\rho'(b)\in U$ such that $$\rho(a^m)\rho'(b^n)=h.$$
\noindent Hence, letting $\rho'(a):=\rho(a)$ proves the claim.

Finally, by Proposition \ref{prop: criterion for non robust qi} the representation $\rho'$ is accumulated by representations containing unipotent elements. Remark \ref{rem: unipotents} yields the result.

\newcommand{\etalchar}[1]{$^{#1}$}


\begin{thebibliography}{GGKW17}

\bibitem[ABY10]{avila-bochi-yoccoz}
Artur Avila, Jairo Bochi, and Jean-Christophe Yoccoz.
\newblock Uniformly hyperbolic finite-valued {{\(\mathrm{SL}(2,\mathbb
  R)\)}}-cocycles.
\newblock {\em Comment. Math. Helv.}, 85(4):813--884, 2010.

\bibitem[Bar10]{barbot}
Thierry Barbot.
\newblock Three-dimensional {Anosov} flag manifolds.
\newblock {\em Geom. Topol.}, 14(1):153--191, 2010.

\bibitem[BPS19]{BPS}
Jairo Bochi, Rafael Potrie, and Andr{\'e}s Sambarino.
\newblock Anosov representations and dominated splittings.
\newblock {\em J. Eur. Math. Soc. (JEMS)}, 21(11):3343--3414, 2019.

\bibitem[BQ16]{benoistquint}
Yves Benoist and Jean-Fran{\c{c}}ois Quint.
\newblock {\em Random walks on reductive groups}, volume~62 of {\em Ergeb.
  Math. Grenzgeb., 3. Folge}.
\newblock Cham: Springer, 2016.

\bibitem[Cor92]{corlette}
Kevin Corlette.
\newblock Archimedean superrigidity and hyperbolic geometry.
\newblock {\em Ann. Math. (2)}, 135(1):165--182, 1992.

\bibitem[Dal11]{dalbogeodesichorocyclic}
Fran{\c{c}}oise Dal'Bo.
\newblock {\em Geodesic and horocyclic trajectories. {Transl}. from the
  {French} by the author}.
\newblock Universitext. Les Ulis: EDP Sciences; Berlin: Springer, 2011.

\bibitem[DGK{\etalchar{+}}23]{DGKLMconvexcoxeter}
Jeffrey Danciger, François Guéritaud, Fanny Kassel, Gye-Seon Lee, and Ludovic
  Marquis.
\newblock Convex cocompactness for coxeter groups.
\newblock 2023.

\bibitem[DH24]{DeyHurfulllimit}
Subhadip Dey and Sebastian Hurtado.
\newblock Remarks on discrete subgroups with full limit sets in higher rank lie
  groups, 2024.

\bibitem[DK06]{deblois}
Jason DeBlois and Richard P.~IV Kent.
\newblock Surface groups are frequently faithful.
\newblock {\em Duke Math. J.}, 131(2):351--362, 2006.

\bibitem[FW07]{wolff}
Louis Funar and Maxime Wolff.
\newblock Non-injective representations of a closed surface group into
  {{\(\text{PSL}(2,\mathbb{R})\)}}.
\newblock {\em Math. Proc. Camb. Philos. Soc.}, 142(2):289--304, 2007.

\bibitem[GGKW17]{GGKW}
Fran{\c{c}}ois Gu{\'e}ritaud, Olivier Guichard, Fanny Kassel, and Anna
  Wienhard.
\newblock Anosov representations and proper actions.
\newblock {\em Geom. Topol.}, 21(1):485--584, 2017.

\bibitem[GW12]{GW}
Olivier Guichard and Anna Wienhard.
\newblock Anosov representations: domains of discontinuity and applications.
\newblock {\em Invent. Math.}, 190(2):357--438, 2012.

\bibitem[Hit92]{hitchin}
N.J. Hitchin.
\newblock Lie groups and teichmüller space.
\newblock {\em Topology}, 31(3):449--473, 1992.

\bibitem[Kap08]{kapkleiniansurvey}
Michael Kapovich.
\newblock Kleinian groups in higher dimensions.
\newblock In {\em Geometry and dynamics of groups and spaces}, volume 265 of
  {\em Progr. Math.}, pages 487--564. Birkh\"auser, Basel, 2008.

\bibitem[Kas18]{kassel}
Fanny Kassel.
\newblock Geometric structures and representations of discrete groups.
\newblock In {\em Proceedings of the international congress of mathematicians
  2018, ICM 2018, Rio de Janeiro, Brazil, August 1--9, 2018. Volume II. Invited
  lectures}, pages 1115--1151. Hackensack, NJ: World Scientific; Rio de
  Janeiro: Sociedade Brasileira de Matem{\'a}tica (SBM), 2018.

\bibitem[Kat92]{skatok}
Svetlana Katok.
\newblock {\em Fuchsian groups}.
\newblock Chicago: The University of Chicago Press, 1992.

\bibitem[KLP17]{KLP}
Michael Kapovich, Bernhard Leeb, and Joan Porti.
\newblock Anosov subgroups: dynamical and geometric characterizations.
\newblock {\em Eur. J. Math.}, 3(4):808--898, 2017.

\bibitem[Lab06]{labourie}
Fran{\c{c}}ois Labourie.
\newblock Anosov flows, surface groups and curves in projective space.
\newblock {\em Invent. Math.}, 165(1):51--114, 2006.

\bibitem[Lah23]{lahn}
Max Lahn.
\newblock Reducible suspensions of anosov representations, 2023.

\bibitem[Mn78]{mane}
Ricardo Ma\~n\'e.
\newblock Contributions to the stability conjecture.
\newblock {\em Topology}, 17:383--396, 1978.

\bibitem[Mor15]{witte-morris}
Dave~Witte Morris.
\newblock {\em Introduction to arithmetic groups}.
\newblock [s.l.]: Deductive Press, 2015.

\bibitem[MR95]{mcshane-rivin}
Greg McShane and Igor Rivin.
\newblock A norm on homology of surfaces and counting simple geodesics.
\newblock {\em Internat. Math. Res. Notices}, (2):61--69, 1995.

\bibitem[Pot18]{potrie}
Rafael Potrie.
\newblock Robust dynamics, invariant structures and topological classification.
\newblock In {\em Proceedings of the international congress of mathematicians
  2018, ICM 2018, Rio de Janeiro, Brazil, August 1--9, 2018. Volume III.
  Invited lectures}, pages 2063--2085. Hackensack, NJ: World Scientific; Rio de
  Janeiro: Sociedade Brasileira de Matem{\'a}tica (SBM), 2018.

\bibitem[Sul85]{sullivan}
Dennis Sullivan.
\newblock Quasiconformal homeomorphisms and dynamics. {II}: {Structural}
  stability implies hyperbolicity for {Kleinian} groups.
\newblock {\em Acta Math.}, 155:243--260, 1985.

\bibitem[Tso23]{tsouvalas}
Konstantinos Tsouvalas.
\newblock Quasi-isometric embeddings inapproximable by {Anosov}
  representations.
\newblock {\em J. Inst. Math. Jussieu}, 22(5):2497--2514, 2023.

\bibitem[Tso24]{tsouvalas2}
Konstantinos Tsouvalas.
\newblock Robust quasi-isometric embeddings inapproximable by anosov
  representations, 2024.

\bibitem[Wie18]{weinhard}
Anna Wienhard.
\newblock An invitation to higher {Teichm{\"u}ller} theory.
\newblock In {\em Proceedings of the international congress of mathematicians
  2018, ICM 2018, Rio de Janeiro, Brazil, August 1--9, 2018. Volume II. Invited
  lectures}, pages 1013--1039. Hackensack, NJ: World Scientific; Rio de
  Janeiro: Sociedade Brasileira de Matem{\'a}tica (SBM), 2018.

\end{thebibliography}
\end{document}